\documentclass[10pt]{article}
\usepackage[a4paper]{geometry}
\usepackage{amsmath,amsfonts,amssymb,amsthm,mathtools,mathrsfs,calc}
\usepackage[english]{babel}

\usepackage[usenames,dvipsnames]{xcolor}
\usepackage{tikz}
\usepackage[all]{xy}
\usepackage{caption}
\usepackage{graphicx}
\usepackage{mathrsfs}

\usepackage{pifont}
\usepackage{shuffle}
\usepackage{enumitem}

\usepackage[colorlinks=true]{hyperref}

\usepackage{color, colortbl}
\definecolor{paleblue}{rgb}{0.7, 0.7, 1.0}
\definecolor{palegreen}{rgb}{0.7, 1,0.7}
\colorlet{colori}{red!80!blue!80!black}
\colorlet{colorj}{red!30!blue!70!brown}
\colorlet{colork}{blue!40!black}
\colorlet{colorM}{olive}

  \theoremstyle{plain}
\newtheorem{theorem}{Theorem}
\newtheorem{proposition}{Proposition}[section]
\newtheorem{corollary}[proposition]{Corollary}

  \theoremstyle{remark}
\newtheorem{remark}[proposition]{Remark}
  \theoremstyle{definition}
\newtheorem{definition}[proposition]{Definition}

\newtheorem{lemma}[proposition]{Lemma}
\newtheorem{example}[proposition]{Example}

\newcommand{\II}{\mathbf{I}}
\newcommand{\NN}{\mathbb{N}}
\newcommand{\ZZ}{\mathbb{Z}}

\newcommand{\CC}{\mathcal C}

\newcommand{\Hom}{\operatorname{Hom}}
\newcommand{\End}{\operatorname{End}}

\newcommand{\Ob}{\operatorname{Ob}}

\renewcommand{\k}{\Bbbk}
\newcommand{\Vect}{\mathbf{Vect}_\k}

\newcommand{\Mod}{\mathbf{Mod}}

\newcommand{\YD}{\mathbf{YD}}
\newcommand{\YDAlg}{\mathbf{YDAlg}}

\newcommand{\Id}{\operatorname{Id}}

\newcommand{\one}{\mathbf{1}}

\newcommand{\oV}{\overline{V}}

\newcommand{\oW}{\overline{W}}

\newcommand{\of}{\overline{f}}
\newcommand{\oxi}{\overline{\xi}}
\newcommand{\osigma}{\overline{\sigma}}

\newcommand{\crot}{\mathring{c}}

\newcommand{\BrSyst}{\mathbf{BrSyst}}
\newcommand{\bialg}{\mathbf{bialg}}
\newcommand{\ibr}{i_{br}}
\renewcommand{\le}{\leqslant}
\renewcommand{\ge}{\geqslant}

\newcommand*{\longhookrightarrow}{\ensuremath{\lhook\joinrel\relbar\joinrel\rightarrow}}

\setlist{nolistsep}
\newcommand*\circled[1]{\tikz[baseline=(char.base)]{
            \node[shape=circle,draw,inner sep=1pt] (char) {#1};}}

\begin{document}

\title{R-Matrices, Yetter-Drinfel$'$d Modules \\ and Yang-Baxter Equation}

\author{Victoria Lebed \\ \itshape lebed@math.jussieu.fr}

\maketitle

\begin{abstract}
\footnotesize
In the first part we recall two famous sources of solutions to the Yang-Baxter equation -- R-matrices and Yetter-Drinfel$'$d (=YD) modules -- and an interpretation of the former as a particular case of the latter. We show that this result holds true in the more general case of weak R-matrices, introduced here. In the second part we continue exploring the ``braided'' aspects of YD module structure, exhibiting a braided system encoding all the axioms from the definition of YD modules. The functoriality and several generalizations of this construction are studies using the original machinery of YD systems.  As consequences, we get a conceptual interpretation of the tensor product structures for YD modules, and a generalization of the deformation cohomology of YD modules. The latter homology theory is thus included into the unifying framework of braided homologies, which contains among others Hochschild, Chevalley-Eilenberg, Gerstenhaber-Schack and quandle homologies.
\end{abstract}


\tableofcontents

\section{Introduction}\label{sec:intro}

The \textit{Yang-Baxter equation} \eqref{eqn:YB} is omnipresent in modern mathematics. Its realm stretches from statistical mechanics to quantum field theory, covering quantum group theory, low-dimensional topology and many other fascinating areas of mathematics and physics. The attemps to understand and classify all solutions to the Yang-Baxter equation (often referred to as \textit{braidings}, since they provide representations of braid groups) have been fruitless so far. Nevertheless we dispose of several methods of producing vast families of such solutions, often endowed with an extremely rich structure.

In Section \ref{sec:R_vs_YD} we review two major algebraic sources of braidings, which have been thouroughly studied from various viewpoints over the past few decades. The first one is given by \textit{R-matrices} for {quasi-triangular bialgebras}, dating back to V.G. Drinfel$'$d's celebrated 1986 ICM talk \cite{DrQuGroups}. The second one comes from \textit{Yetter-Drinfel$'$d modules} (= YD modules)  over a bialgebra, introduced by D. Yetter in 1990 under the name of ``crossed bimodules'' (see \cite{Yetter}) and rediscovered later by different authors under different names (see for instance the paper \cite{Wor} of S.L. Woronowicz, where he implicitely considers a Hopf algebra as a YD module over itself). All these notions and constructions are recalled in detail in Sections \ref{sec:R} and \ref{sec:YD}.

Note that we are interested here in \textbf{not necessarily invertible} braidings; that is why most constructions are effectuated over \textit{bialgebras} rather than \textit{Hopf algebras}, though the latter are more current in literature. Section \ref{sec:YD_br_systems} contains an example where the non-invertibility does matter.

It was shown in S. Montgomery's famous book (\cite{Montgomery}, 10.6.14) that \textbf{R-matrix solutions to the YBE can be interpreted as particular cases of Yetter-Drinfel$'$d type solutions}. We recall this result and its categorical version (due to M. Takeuchi, cf. \cite{Takeuchi}) in Section \ref{sec:inclusion}. Our original contribution consists in a generalization of this result: we introduce the notion of \textit{weak R-matrix} for a bialgebra $H$ and show it to be sufficient for endowing any $H$-module with a YD module structure over $H$ (cf. the charts on p. \pageref{charts}).

In Section \ref{sec:YD_YB} we explore \textbf{deeper connections between YD modules} (and thus R-matrices, as explained above) \textbf{and the YBE}. We show that YD modules give rise not only to braidings, but also to higher-level braided structures, called \textit{braided systems}. We will briefly explain this concept after a short reminder on a ``local'' and a ``global'' category-theoretic viewpoints on the YBE. 

The first one is rather straightforward: in a strict monoidal category $\CC,$ one looks for objects $V$  and endomorphisms $\sigma$ of $V \otimes V$ satisfying \eqref{eqn:YB}. Such $V$'s are called \textit{braided objects} in $\CC.$ A more categorical approach consists in working in a ``globally'' \textit{braided monoidal category}, as defined in 1993 by A. Joyal and R.H. Street (\cite{BraidedCat}). Concretely, a braiding on a monoidal category is a natural family of morphisms $\sigma_{V,W}:V \otimes W \rightarrow W \otimes V$ compatible with the monoidal structure, in the sense of \eqref{eqn:br_cat}-\eqref{eqn:br_cat2}. Every object $V$ of such a category is braided, via $\sigma_{V,V}.$ See \cite{Takeuchi} for a comparison of ``local'' and  ``global'' approaches, in particular in the the context of the definition of braided Hopf algebras. Famous braidings on the category of modules over a quasi-triangular bialgebra and in that of YD modules over a bialgebra are recalled in Sections \ref{sec:R} and \ref{sec:YD}.

Now, the notion of braided system in $\CC$ is a multi-term version of that of braided object: it is a family $V_1, \ldots, V_r$ of objects in $\CC$ endowed with morphisms $\sigma_{i,j}:V_i \otimes V_j \rightarrow V_j \otimes V_i $ for all $ i  \le j,$ satisfying a system of mixed YBEs.  This notion was defined and studied
in \cite{Lebed} (see also \cite{Lebed2}). The $r=2$ case appeared, under the name of \textit{Yang-Baxter system} (or \textit{WXZ system}), in the work of L. Hlavat{\'y} and L. {\v{S}}nobl (\cite{HlSn}). See  Section \ref{sec:br_systems} for details.

In Section \ref{sec:YD_br_systems}, we present a \textbf{braided system structure on the family $(H,M,H^*)$} for a YD module $M$ over a finite-dimensional $\k$-bialgebra $H$ (here $\k$ is a field, and $H^*$ is the linear dual of $H$). 
 Note that several $\sigma_{i,j}$'s from this system are highly non-invertible. Its subsystem $(H,H^*)$ is the braided system encoding the bialgebra structure, constructed and explored in \cite{Lebed2}. The latter construction is related to, but different from, the braided system considered by F.F. Nichita in \cite{Ni_bialg} (see also \cite{YBSyst_Entwining}). 
 
In order to treat the above construction in a conceptual way and extend it to a braided system structure on $(H,V_1,\ldots,V_s,H^*)$ for YD modules $V_1,\ldots,V_s$ over $H,$ we introduce the concept of \textit{Yetter-Drinfel$'$d system} and show it to be automaticaly endowed with a braiding. See Section \ref{sec:YD_systems} for details, and  Section \ref{sec:YD_systems_ex} for examples.

 In Section \ref{sec:YD_systems_prop} we show the {functoriality} and the {precision} of the above braided system construction. The \textit{functoriality} is proved by exhibiting the category inclusion \eqref{eqn:cat_incl_YD_BrSyst}. \textit{Precision} means that the described braided system  $(H,M,H^*)$ captures all the algebraic information about the YD module $M,$ in the sense that each mixed YBE for this system is equivalent to an axiom from the definition of a YD module, and each YD axiom gets a ``braided'' interpretation in this way.
 
Section \ref{sec:YD_tensor} contains an unexpected application of the braided system machinery. Namely, it allows us to recover the two \textit{tensor product structures for YD modules}, proposed by L.A. Lambe and D.E. Radford in \cite{Rad}, from a conceptual viewpoint. 
 
 Applying the general \textit{braided (co)homology} theory from \cite{Lebed1} and \cite{Lebed2} (recalled in Section \ref{sec:hom}) to the braided systems above, we obtain in Section \ref{sec:YD_hom} a rich \textbf{(co)homology theory for (families of) YD modules}. It contains in particular the \textit{deformation cohomology of YD modules}, introduced by F. Panaite and D. {\c{S}}tefan in \cite{Panaite}. The results of Section \ref{sec:inclusion} give then for free a braided system structure for any module over a finite-dimensional quasi-triangular $\k$-bialgebra $H,$ with a corresponding (co)homology theory. Besides a generalization of the Panaite-{\c{S}}tefan construction, our ``braided'' tools allow to considerably simplify otherwise technical verifications from their theory. Moreover, our approach allows to consider YD module (co)homologies in the same \textit{unifying framework} as the (co)homologies of other algebraic structures admitting a braided interpretation, e.g. associative and Leibniz (or Lie) algebras, self-distributive structures, bialgebras and Hopf (bi)modules (see \cite{Lebed1} and \cite{Lebed2}).

\medskip
The paper is intended to be as elementary and as self-contained as possible. Even widely known notions are recalled for the reader's convenience. The already classical \textit{graphical calculus} is extensively used in this paper, with
\begin{itemize}
\item dots standing for vector spaces (or objects in a monoidal category),
\item horizontal gluing corresponding to the tensor product,
\item graph diagrams representing morphisms from the vector space (or object) which corresponds to the lower dots to that corresponding to the upper dots,
\item vertical gluing standing for morphism composition, and vertical strands for identities.
\end{itemize}
Note that all diagrams in this paper are to be read \textbf{from bottom to top}.

Throughout this paper we work in a \textit{strict monoidal category} $\CC$ (Definition \ref{def:monoidal_cat}); as an example, one can have in mind the category $\Vect$ of $\k$-vector spaces and $\k$-linear maps, endowed with the usual tensor product over $\k.$

For an object $V$ in $\CC$ and a morphism $\varphi:V^{\otimes l}\rightarrow V^{\otimes r},$ the following notation is repeatedly used:
\begin{equation}\label{eqn:phi_i}
\varphi^i := \Id_V^{\otimes (i-1)}\otimes\varphi \otimes \Id_V^{\otimes (k-i+1)}:V^{\otimes (k+l)}\rightarrow V^{\otimes(k+r)}.
\end{equation}
Similar notations are used for morphisms on tensor products of different objects.

\section{Two sources of braidings revisited}\label{sec:R_vs_YD}

\subsection{Basic definitions}\label{sec:basic_def}

\begin{definition}\label{def:monoidal_cat}
A \emph{strict monoidal (or tensor) category} is a category $\CC$ endowed with
\begin{itemize}
\item a \emph{tensor product} bifunctor $\otimes: \CC\times \CC \rightarrow \CC$ satisfying the associativity condition;
\item a \emph{unit} object $\II$ which is a left and right identity for $\otimes.$
\end{itemize}
\end{definition}

We work only with {\textbf{strict}} monoidal categories here for the sake of simplicity; according to a theorem of S. MacLane (\cite{Cat}), any monoidal category is monoidally equivalent to a strict one. This justifies in particular parentheses-free notations like $V\otimes W\otimes U$ or $V^{\otimes n}.$ The word ``strict'' is omitted but always implied in what follows.

The ``local'' categorical notion of braiding will be extensively used in this paper:

\begin{definition}\label{def:LocalBr}
An \emph{object} $V$ in a monoidal category $\CC$ is called \emph{braided} if it is endowed with a ``local'' \emph{braiding}, i.e. a morphism
 $$\sigma=\sigma_{V} : V\otimes V \rightarrow V \otimes V,$$
satisfying the (categorical) \emph{Yang-Baxter equation} (= \emph{YBE})
\begin{equation}\label{eqn:YB}\tag{YBE}
(\sigma_{V}\otimes \Id_V)\circ(\Id_V \otimes \sigma_{V})\circ(\sigma_{V}\otimes \Id_V) =(\Id_V \otimes \sigma_{V}
)\circ(\sigma_{V}\otimes \Id_V)\circ(\Id_V \otimes \sigma_{V}).
\end{equation}
\end{definition}

Following D. Yetter (\cite{Yetter}), one should use the term \emph{pre-braiding} here in order to stress that non-invertible $\sigma$'s are allowed; we keep the term \emph{braiding} for simplicity.

Graphically, the braiding $\sigma_{V}$ is presented as $\:$ 
\begin{tikzpicture}[scale=0.4]
\draw [ rounded corners](0,0)--(0,0.25)--(0.4,0.4);
\draw [ rounded corners](0.6,0.6)--(1,0.75)--(1,1);
\draw [ rounded corners](1,0)--(1,0.25)--(0,0.75)--(0,1);
\node  at (2,0)  {.};
\end{tikzpicture}
The diagrammatical counterpart of \eqref{eqn:YB}, depicted on Fig. \ref{pic:YB}, is then the \emph{third Reidemeister move}, which is at the heart of knot theory.
\begin{center}
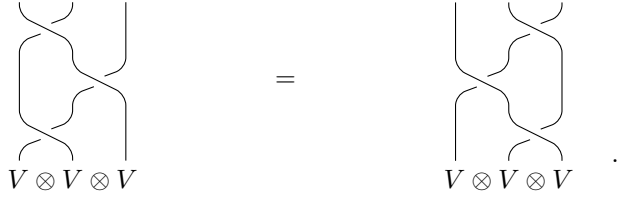

\begin{tikzpicture}[scale=0.7]
\draw [ rounded corners](0,0)--(0,0.25)--(0.4,0.4);
\draw [ rounded corners](0.6,0.6)--(1,0.75)--(1,1.25)--(1.4,1.4);
\draw [ rounded corners](1.6,1.6)--(2,1.75)--(2,3);
\draw [ rounded corners](1,0)--(1,0.25)--(0,0.75)--(0,2.25)--(0.4,2.4);
\draw [ rounded corners](0.6,2.6)--(1,2.75)--(1,3);
\draw [ rounded corners](2,0)--(2,1.25)--(1,1.75)--(1,2.25)--(0,2.75)--(0,3);
\node  at (1,0) [below] {$V \otimes V \otimes V$};
\node  at (5,1.5){$=$};
\end{tikzpicture}
\begin{tikzpicture}[scale=0.7]
\node  at (-2.5,1.5){};
\draw [ rounded corners](1,1)--(1,1.25)--(1.4,1.4);
\draw [ rounded corners](1.6,1.6)--(2,1.75)--(2,3.25)--(1,3.75)--(1,4);
\draw [ rounded corners](0,1)--(0,2.25)--(0.4,2.4);
\draw [ rounded corners](0.6,2.6)--(1,2.75)--(1,3.25)--(1.4,3.4);
\draw [ rounded corners](1.6,3.6)--(2,3.75)--(2,4);
\draw [ rounded corners](2,1)--(2,1.25)--(1,1.75)--(1,2.25)--(0,2.75)--(0,4);
\node  at (1,1) [below] {$V \otimes V \otimes V$};
\node  at (3,1){.};
\end{tikzpicture}
\captionof{figure}{Yang-Baxter equation $\longleftrightarrow$ Reidemeister move III}\label{pic:YB}
\end{center} 

\medskip
The ``global'' categorical notion of braiding will also be used here, both for describing the underlying category $\CC$ and for constructing and systematizing new braidings:

\begin{definition}\label{def:GlobalBr}
\begin{itemize}
\item
A monoidal category $\CC$ is called \emph{braided} if it is endowed with a \emph{braiding (or a commutativity constraint)}, i.e. a natural family of morphisms
 $$c= (c_{V,W} : V\otimes W \rightarrow W \otimes V )_{V,W \in \Ob(\CC)},$$
 satisfying
 \begin{align}
 c_{V,W\otimes U} &=(\Id_W \otimes c_{V,U} )\circ(c_{V,W}\otimes \Id_U),\label{eqn:br_cat}\\
 c_{V\otimes W, U} &=(c_{V,U}\otimes \Id_W)\circ(\Id_V \otimes c_{W,U})\label{eqn:br_cat2}
 \end{align}
for any triple of objects $V, W, U.$ ``Natural'' means here
\begin{equation}\label{eqn:Nat}
c_{V',W'} \circ(f\otimes g) = (g\otimes f)\circ c_{V,W}
\end{equation}
 for all $ V,W,V',W' \in \Ob(\CC), f \in \Hom_\CC(V,V'), g \in \Hom_\CC(W,W').$
 
\item A braided category $\CC$ is called \emph{symmetric} if its braiding is symmetric:
\begin{align}\label{eqn:BrSymm}
c_{V,W} \circ c_{W,V} &=\Id_{W\otimes V} & \forall\: V,W \in \Ob(\CC).
\end{align}
\end{itemize}
\end{definition}

 The part ``monoidal'' of the usual terms ``braided monoidal'' and ``symmetric monoidal'' are omitted in what follows.

\begin{lemma}\label{thm:gl_br_is_loc}
Every object $V$ in a braided category $\CC$ is braided, with $\sigma_{V}=c_{V,V}.$ 
\end{lemma}

\begin{proof}
Take $V'=V,$ $W'=W=V \otimes V,$ $f=\Id_V$ and $g=c_{V,V}$ in the condition \eqref{eqn:Nat} expressing naturality; this gives the YBE. 
\end{proof}

\medskip 
 We further define the structures of algebra, coalgebra, bialgebra and Hopf algebra in a monoidal category $\CC$:

\begin{definition}\label{def:AlgCat}
\begin{itemize}
 \item A \emph{unital associative algebra} (= \emph{UAA}) in $\CC$ is an object $A$ together with morphisms $\mu:A\otimes A\rightarrow A$ and $\nu:\II\rightarrow A$ satisfying  the associativity and the unit conditions:
 \begin{align*}
  \mu\circ(\Id_A\otimes\mu) &=\mu\circ(\mu \otimes \Id_A):A^{\otimes 3}\rightarrow A,\\
  \mu\circ(\Id_A\otimes\nu) &=\mu\circ(\nu \otimes \Id_A) = \Id_A.
 \end{align*}
 A UAA morphism $\varphi$ between UAAs $(A,\mu_A,\nu_A)$ and $(B,\mu_B,\nu_B)$ is a $\varphi \in \Hom_\CC(A,B)$ respecting the UAA structures:
 \begin{align}
 \varphi \circ \mu_{A} &= \mu_{B} \circ (\varphi \otimes \varphi): A \otimes A \rightarrow B,\label{eqn:respect_mu}\\
 \varphi \circ \nu_{A} &= \nu_{B}.\label{eqn:respect_nu}
 \end{align} 
 
\item A UAA $(A,\mu,\nu)$ in $\CC$ is called \emph{braided} if it is endowed with a braiding $\sigma$ compatible with the UAA structure. Using notation \eqref{eqn:phi_i}, this can be written as
\begin{align}
\sigma \circ\mu^1 &= \mu^2\circ (\sigma^1 \circ \sigma^2):A^{\otimes 3} \rightarrow A^{\otimes 2},\label{eqn:BrAlg}\\
\sigma \circ\mu^2 &= \mu^1\circ (\sigma^2 \circ \sigma^1):A^{\otimes 3} \rightarrow A^{\otimes 2};\label{eqn:BrAlg'}\\
\sigma \circ\nu^1 &= \nu^2:A = \II \otimes A = A \otimes \II \rightarrow A^{\otimes 2},\label{eqn:BrAlg''}\\
\sigma \circ\nu^2 &= \nu^1:A = \II \otimes A = A \otimes \II \rightarrow A^{\otimes 2}.\label{eqn:BrAlg'''}
\end{align} 

 \item A \emph{counital coassociative coalgebra} (= \emph{coUAA}) in $\CC$ is an object $C$ together with morphisms $\Delta:C\rightarrow C\otimes C$ and $\varepsilon:C \rightarrow \II$ satisfying  the coassociativity and the counit conditions:
 \begin{align*}
 (\Delta \otimes \Id_C) \circ \Delta &= (\Id_C \otimes \Delta) \circ \Delta: C \rightarrow C^{\otimes 3},\\
 (\varepsilon \otimes \Id_C) \circ \Delta &= (\Id_C \otimes \varepsilon) \circ \Delta = \Id_C.
 \end{align*}
 A coUAA morphism $\varphi$ between coUAAs is a morphism in $\CC$ respecting the coUAA structures.
 
\item A coUAA $(C,\Delta,\varepsilon)$ in $\CC$ is called \emph{braided} if it is endowed with a braiding $\sigma$ compatible with the coUAA structure, in the sense analogous to \eqref{eqn:BrAlg}-\eqref{eqn:BrAlg'''}.

 \item A \emph{left module} over a UAA $(A,\mu,\nu)$ in $\CC$ is an object $M$ together with a morphism $\lambda:A\otimes M\rightarrow M$ respecting $\mu$ and $\nu$: 
 \begin{align}
   \lambda \circ(\mu \otimes \Id_M) &= \lambda \circ(\Id_A\otimes \lambda) :A \otimes A \otimes M\rightarrow M,\label{eqn:Amod}\\
  \lambda \circ(\nu \otimes \Id_M) &= \Id_M. \label{eqn:Amod'}
 \end{align}
 An $A$-module morphism $\varphi$ between $A$-modules $(M,\lambda_M)$ and $(N,\lambda_N)$ is a $\varphi \in \Hom_\CC(M,N)$ respecting the $A$-module structures:
 \begin{align}\label{eqn:AmodMor}
 \varphi \circ \lambda_M &= \lambda_N \circ (\Id_A \otimes \varphi): A \otimes M \rightarrow N.
 \end{align}
 The category of $A$-modules in $\CC$ and their morphisms is denoted by $_A\!\Mod.$ 
 
 Right modules, left/right comodules over coUAAs and their morphisms are defined similarly.

\item 
A \emph{bialgebra} in a \underline{braided} category 
 $(\CC,\otimes,\II,c)$ is a UAA structure $(\mu,\nu)$ and a coUAA structure $(\Delta,\varepsilon)$ on an object $H,$ compatible in the following sense:
\begin{align}
\Delta \circ \mu &= (\mu \otimes \mu)\circ ( \Id_H \otimes c_{H,H} \otimes \Id_H) \circ (\Delta \otimes \Delta): H \otimes H \rightarrow  H \otimes H, \label{eqn:cat_bialg}\\
\Delta \circ \nu &= \nu \otimes \nu : \II \rightarrow  H \otimes H,\notag\\
\varepsilon \circ \mu &= \varepsilon \otimes \varepsilon: H \otimes H \rightarrow  \II, \notag\\
\varepsilon \circ \nu &= \Id_{\II}: \II \rightarrow  \II.\notag
\end{align}
A bialgebra morphism is a morphism which respects UAA and coUAA structures simultaneously.

\item
If moreover $H$ has an \emph{antipode}, i.e. a morphism $s:H\rightarrow H$ satisfying 
\begin{equation}\label{eqn:s}\tag{s}
\mu \circ (s\otimes \Id_H) \circ \Delta = \mu \circ (\Id_H \otimes s) \circ \Delta  = \nu \circ \varepsilon,
\end{equation}
then it is called a \emph{Hopf algebra} in $\CC.$
\end{itemize}
\end{definition} 
 
  The notions of (braided) algebra and coalgebra, and of module and comodule, are mutually dual, while that of braiding, of bialgebra and Hopf algebra are self-dual; see \cite{Cat} or \cite{Lebed} for more details on the \textit{categorical duality}. Graphically, applying this duality consists simply in turning all the diagrams upside down, i.e. taking a
\textit{horizontal mirror image}. Fig. \ref{pic:AssCoass} contains for instance the graphical depictions of the associativity and the coassociativity axioms. Here and afterwards a multiplication $\mu$ is represented as \begin{tikzpicture}[scale=0.3]
 \draw(0,1) -- (0,0)-- (1,-1);
 \draw (0,0)-- (-1,-1);
 \fill[teal] (0,0) circle (0.2);
\end{tikzpicture}, and a comultiplication $\Delta$ -- as \begin{tikzpicture}[scale=0.3]
 \draw(0,-1) -- (0,0)-- (1,1);
 \draw (0,0)-- (-1,1);
 \fill[teal] (0,0) circle (0.2);
\end{tikzpicture}.

\begin{center}
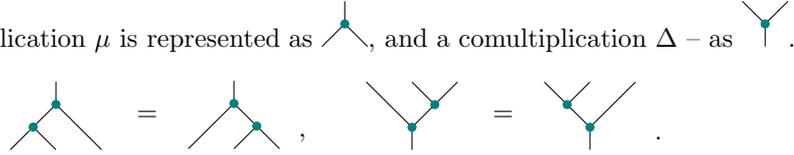

\begin{tikzpicture}[scale=0.3]
\draw (0,0)--(2,2);
\draw (4,0)--(2,2);
\draw (1,1)--(2,0);
\draw (2,3)--(2,2);
\fill[teal] (1,1) circle (0.2);
\fill[teal] (2,2) circle (0.2);
\node at (6,1.5) {$=$};
\node at (7,1.5) {};
\end{tikzpicture}
\begin{tikzpicture}[scale=0.3]
\draw (0,0)--(2,2);
\draw (4,0)--(2,2);
\draw (3,1)--(2,0);
\draw (2,3)--(2,2);
\fill[teal] (3,1) circle (0.2);
\fill[teal] (2,2) circle (0.2);
\node at (5,0.5) {,};
\node at (7,1) {$ $};
\end{tikzpicture}
\begin{tikzpicture}[scale=0.3]
\draw (0,0)--(2,-2);
\draw (4,0)--(2,-2);
\draw (3,-1)--(2,0);
\draw (2,-3)--(2,-2);
\fill[teal] (3,-1) circle (0.2);
\fill[teal] (2,-2) circle (0.2);
\node at (6,-1.5) {$=$};
\node at (7,-1.5) {};
\end{tikzpicture}
\begin{tikzpicture}[scale=0.3]
\draw (0,0)--(2,-2);
\draw (4,0)--(2,-2);
\draw (1,-1)--(2,0);
\draw (2,-3)--(2,-2);
\fill[teal] (1,-1) circle (0.2);
\fill[teal] (2,-2) circle (0.2);
\node at (5,-2.5) {.};
\end{tikzpicture}
   \captionof{figure}{Associativity and coassociativity}\label{pic:AssCoass}
\end{center}   

Graphical versions of several other axioms from the above definition are presented on Figs \ref{pic:BrCoalg} and \ref{pic:Bialg}.
  \begin{center}
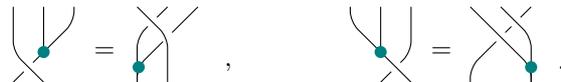

\begin{tikzpicture}[scale=0.4]
 \draw[rounded corners](0.65,0.65) -- (2,2)-- (2,2.5);
 \draw (0,0) -- (0.35,0.35);
 \draw (1,1) -- (1,2.5);
 \draw[rounded corners](1,0) -- (0,1)-- (0,2.5);
 \fill[teal] (1,1) circle (0.2);
 \node at (3,1) {$=$};
\end{tikzpicture}
\begin{tikzpicture}[scale=0.4]
 \draw[rounded corners](0,-0.5) -- (0,0)--(0.85,0.85);
 \draw (1.15,1.15) --   (2,2);
 \draw[rounded corners](0,-0.5) -- (0,1)--(0.35,1.35);
 \draw (0.65,1.65) --  (1,2);
 \draw[rounded corners] (1,-0.5) -- (1,1) -- (0,2);
 \fill[teal] (0.05,0) circle (0.2);
 \node at (3,0) {,};
\end{tikzpicture}
\begin{tikzpicture}[scale=0.4]
 \node at (-3,1) {};
 \draw[rounded corners](1.65,0.65) -- (2,1)-- (2,2.5);
 \draw (1,0) -- (1.35,0.35);
 \draw (1,1) -- (1,2.5);
 \draw[rounded corners](2,0) -- (0,2)-- (0,2.5);
 \fill[teal] (1,1) circle (0.2);
 \node at (3,1) {$=$};
\end{tikzpicture}
\begin{tikzpicture}[scale=0.4]
 \draw[rounded corners](0,-0.5) -- (0,0)-- (0.85,0.85);
 \draw (1.15,1.15) -- (1.35,1.35);
 \draw (1.65,1.65) -- (2,2);
 \draw[rounded corners](2,-0.5) -- (2,1)-- (1,2);
 \draw[rounded corners](2,-0.5) -- (2,0)-- (0,2);
 \fill[teal] (2,0) circle (0.2);
 \node at (3,0) {.};
\end{tikzpicture}
\captionof{figure}{Compatibility conditions for a braiding and a comultiplication}\label{pic:BrCoalg}
 \end{center}

\begin{center}
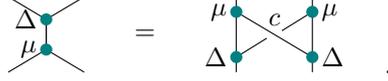

\begin{tikzpicture}[scale=1]
 \draw (0,0) -- (0.5,0.3) -- (0.5,0.7) -- (0,1);
 \draw (1,0) -- (0.5,0.3) -- (0.5,0.7) -- (1,1);
 \node at (0.5,0.3) [left] {$\mu$};
 \node at (0.5,0.7) [left] {$\Delta$};
 \fill [teal] (0.5,0.3) circle (0.08);
 \fill [teal] (0.5,0.7) circle (0.08);
 \node at (1.8,0.5) {$=$};
 \draw (3,0) -- (3,1);
 \draw (4,1) -- (4,0);
 \draw (3,0.2) -- (3.4,0.45);
 \draw (3.6,0.55) -- (4,0.8);
 \draw (3,0.8) -- (4,0.2);
 \fill [teal] (3,0.2) circle (0.08);
 \fill [teal] (3,0.8) circle (0.08);
 \fill [teal] (4,0.8) circle (0.08);
 \fill [teal] (4,0.2) circle (0.08);
 \node at (3,0.8) [left] {$\mu$};
 \node at (3,0.2) [left] {$\Delta$};
 \node at (4,0.8) [right] {$\mu$};
 \node at (4,0.2) [right] {$\Delta$};
 \node at (3.5,0.5) [above] {$c$};
 \node at (5,0) {.};
\end{tikzpicture}
   \captionof{figure}{Main bialgebra axiom \eqref{eqn:cat_bialg}}\label{pic:Bialg}
\end{center}

Note that \eqref{eqn:cat_bialg} is the only bialgebra axiom requiring a braiding on the underlying category $\CC.$

\subsection{R-matrices}\label{sec:R}

From now on we work in a \underline{symmetric} category  $(\CC,\otimes,\II,c).$ Fix a \underline{bialgebra} $H$ in $\CC.$ A bialgebra structure on $H$ is precisely what is needed for the category $_H\!\Mod$ of its left modules to be monoidal: the tensor product $M \otimes N$ of $H$-modules $(M,\lambda_M)$ and $(N,\lambda_N)$ is endowed with the $H$-module structure
\begin{equation}\label{eqn:tensor_mod}
\lambda_{M \otimes N} := (\lambda_M \otimes \lambda_N) \circ (\Id_H \otimes c_{H,M} \otimes \Id_N)\circ (\Delta \otimes \Id_{M \otimes N}),
\end{equation}
and the unit object $\II$ is endowed with the $H$-module structure
\begin{equation}\label{eqn:unit_mod}
\lambda_{\II} := \varepsilon: H \otimes \II = H \rightarrow  \II.
\end{equation}
In what follows we always assume this monoidal structure on $_H\!\Mod.$

If one wants the category $_H\!\Mod$ to be braided (and thus to provide solutions to the Yang-Baxter equation), an additional \textit{quasi-triangular} structure should be imposed on $H.$ The growing interest to quasi-triangular structures can thus be partially explained by their capacity to produce highly non-trivial solutions to the YBE. The most famous example is given by \textit{quantum groups} (see for instance \cite{Kassel}), which will not be discussed here.

\begin{definition}\label{def:R} 
A bialgebra $H$ in $\CC$ is called \emph{quasi-triangular} if it is endowed with an \emph{R-matrix}, i.e. a morphism $R: \II \rightarrow H \otimes H$ satisfying the following conditions:
 \begin{enumerate}
\item $(\Id_H \otimes \Delta) \circ R = (\mu^{op} \otimes \Id_{H\otimes H}) \circ c^2 \circ (R \otimes R),$
\item $(\Delta \otimes \Id_H) \circ R = (\Id_{H\otimes H} \otimes \mu) \circ c^2 \circ (R \otimes R),$
\item $\mu_{H\otimes H} \circ (R \otimes \Delta) = \mu_{H\otimes H} \circ (\Delta^{op} \otimes R),$
\end{enumerate}
where $c^2$ is a shorthand notation for $\Id_H \otimes c_{H,H} \otimes \Id_H,$
\begin{equation}\label{eqn:mu2}
\mu_{H\otimes H}:= (\mu \otimes \mu) \circ c^2: (H \otimes H) \otimes (H \otimes H) \rightarrow H \otimes H 
\end{equation}
 is the standard multiplication on the tensor product of two UAAs, and
$$\mu^{op}:=\mu \circ c_{H,H}, \qquad \qquad \qquad   \Delta^{op}:= c_{H,H} \circ \Delta$$ are the \emph{twisted} multiplication and comultiplication respectively.

The R-matrix $R$ is called \emph{invertible} if there exists a morphism $R^{-1}: \II \rightarrow H \otimes H$  such that
\begin{equation}\label{eqn:inv_of_R}
\mu_{H \otimes H} \circ (R \otimes R^{-1}) = \mu_{H \otimes H} \circ (R^{-1} \otimes R) = \nu \otimes \nu.
\end{equation}
\end{definition}

Fig. \ref{pic:RMatrix} shows a graphical version of the conditions from the definition.

\begin{center}
\begin{tikzpicture}[scale=0.5]
 \draw [rounded corners] (-1,2) -- (-1,1) -- (0,0);
 \draw  (1,1) -- (0,0);
 \draw  [rounded corners] (1,1) -- (0.5,1.5) -- (0.5,2);
 \draw  [rounded corners] (1,1) -- (1.5,1.5) -- (1.5,2);
 \node at (2.5,1)  {$=$};
 \fill [teal]  (1,1) circle (0.15);
 \fill [palegreen] (0,0) circle (0.15);
 \draw (0,0) circle (0.15);
 \node at (0,0) [below] {$R$};
 \node at (1,1) [below] {$\Delta$};
\end{tikzpicture}
\begin{tikzpicture}[scale=0.5]
 \draw [rounded corners] (1,2) -- (1,1) -- (0,0);
 \draw [rounded corners] (2,2) -- (2,1) -- (1,0);
 \draw [rounded corners] (-1,2) -- (-1,1) --(0,0);
 \draw [rounded corners] (-1,1) -- (-1.3,0.75) -- (1,0);
 \node at (3.5,0)  {,};
 \fill [teal]  (-1,1) circle (0.15);
 \fill [palegreen] (0,0) circle (0.15);
 \draw (0,0) circle (0.15);
 \node at (0,0) [below] {$R$};
 \fill [palegreen] (1,0) circle (0.15);
 \draw (1,0) circle (0.15);
 \node at (1,0) [below] {$R$};
 \node at (-1,1) [below left] {$\mu$};
\end{tikzpicture}
\begin{tikzpicture}[scale=0.5]
 \node at (-2,1)  {};
 \draw [rounded corners] (1,2) -- (1,1) -- (0,0);
 \draw  (-1,1) -- (0,0);
 \draw  [rounded corners] (-1,1) -- (-0.5,1.5) -- (-0.5,2);
 \draw  [rounded corners] (-1,1) -- (-1.5,1.5) -- (-1.5,2);
 \node at (2,1)  {$=$};
 \fill [teal]  (-1,1) circle (0.15);
 \fill [palegreen] (0,0) circle (0.15);
 \draw (0,0) circle (0.15);
 \node at (0,0) [below] {$R$};
 \node at (-1,1) [below] {$\Delta$};
\end{tikzpicture}
\begin{tikzpicture}[scale=0.5]
 \draw [rounded corners] (-1,2) -- (-1,1) -- (0,0);
 \draw [rounded corners] (-2,2) -- (-2,1) -- (-1,0);
 \draw  (1,2) -- (1,1) --(0,0);
 \draw  (1,1) -- (-1,0);
 \node at (2.5,0)  {,};
 \fill [teal]  (1,1) circle (0.15);
 \fill [palegreen] (0,0) circle (0.15);
 \draw (0,0) circle (0.15);
 \node at (0,0) [below] {$R$};
 \fill [palegreen] (-1,0) circle (0.15);
 \draw (-1,0) circle (0.15);
 \node at (-1,0) [below] {$R$};
 \node at (1,1) [below] {$\mu$};
\end{tikzpicture}
\begin{tikzpicture}[scale=0.6]
 \node at (-1,1)  {};
 \draw [rounded corners] (2,2) -- (2,1.3) -- (2,0.5) -- (1.5,0) -- (1.5,-0.5);
 \draw [rounded corners] (1,2) -- (1,1.3) -- (1,0.5) -- (1.5,0);
 \draw  (0,1) -- (1,1.5);
 \draw  (0,1) -- (2,1.5);
 \node at (3.5,1)  {$=$};
 \fill [teal]  (1.5,0) circle (0.15);
 \fill [teal]  (1,1.5) circle (0.15);
 \fill [teal]  (2,1.5) circle (0.15);
 \fill [palegreen] (0,1) circle (0.15);
 \draw (0,1) circle (0.15);
 \node at (0,1) [below] {$R$};
 \node at (1.5,0) [left] {$\Delta$};
 \node at (1,1.5) [above right] {$\mu$};
 \node at (2,1.5) [above right] {$\mu$};
\end{tikzpicture}
\begin{tikzpicture}[scale=0.6]
 \node at (0.5,1)  {};
 \draw [rounded corners] (2,2) -- (2,1.3) -- (1,0.5) -- (1.5,0) -- (1.5,-0.5);
 \draw [rounded corners] (1,2) -- (1,1.3) -- (2,0.5) -- (1.5,0);
 \draw  (3,1) -- (1,1.5);
 \draw  (3,1) -- (2,1.5);
 \node at (3.5,0)  {.};
 \fill [teal]  (1.5,0) circle (0.15);
 \fill [teal]  (1,1.5) circle (0.15);
 \fill [teal]  (2,1.5) circle (0.15);
 \fill [palegreen] (3,1) circle (0.15);
 \draw (3,1) circle (0.15);
 \node at (3,1) [below] {$R$};
 \node at (1.5,0) [left] {$\Delta$};
 \node at (1,1.5) [above left] {$\mu$};
 \node at (2,1.5) [above left] {$\mu$};
\end{tikzpicture}
   
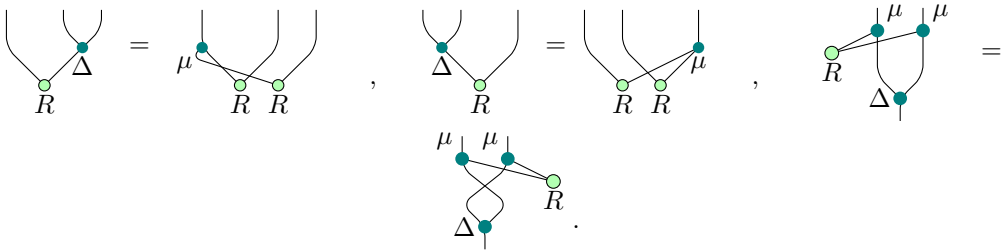
\captionof{figure}{Axioms for an R-matrix}\label{pic:RMatrix}
\end{center}

A well-known result affirms that a quasi-triangular bialgebra structure on $H$ is precisely what is needed for its module category to be braided:

\begin{theorem}\label{thm:R_br_global}
The category $_H\!\Mod$ of left modules over a quasi-triangular bialgebra $H$ in $\CC$ can be endowed with the following braiding (cf. Fig. \ref{pic:RBr}):
\begin{equation}\label{eqn:R_br}
c^R_{M,N}:= c_{M,N} \circ  (\lambda_M \otimes \lambda_N) \circ (\Id_H \otimes c_{H,M} \otimes \Id_N)\circ (R \otimes \Id_{M \otimes N}).
\end{equation}
Here $R$ is the R-matrix of $H,$ and $c$ is the underlying symmetric braiding of the category $\CC.$

If the R-matrix is moreover invertible, then the braiding $c^R$ is invertible as well, with 
\begin{equation}\label{eqn:R_br_inv}
(c^R_{M,N})^{-1}:= (\lambda_M \otimes \lambda_N) \circ (\Id_H \otimes c_{H,M} \otimes \Id_N)\circ (R^{-1} \otimes  c_{N,M}).
\end{equation}
\end{theorem}

\begin{center}
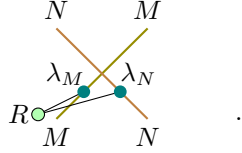

\begin{tikzpicture}[scale=1.2]
 \draw[colorM, thick] (0,0) -- (1,1);
 \draw[brown, thick] (1,0) -- (0,1);
 \draw (-0.2,0.05) -- (0.7,0.3);
 \draw (-0.2,0.05) -- (0.3,0.3);
 \node at (0.9,0.5) {$\lambda_N$};
 \node at (0.1,0.5) {$\lambda_M$}; 
 \node at (-0.2,0.05) [left]{$R$}; 
 \node at (1,0) [below] {$N$};
 \node at (0,0) [below] {$M$};
 \node at (2,0) {.};
 \node at (0,1) [above] {$N$}; 
 \node at (1,1) [above] {$M$};
 \fill[teal] (0.7,0.3) circle (0.07);
 \fill[teal] (0.3,0.3) circle (0.07);
 \fill [palegreen] (-0.2,0.05) circle (0.07);
 \draw (-0.2,0.05) circle (0.07);
\end{tikzpicture}
   \captionof{figure}{A braiding for $H$-modules}\label{pic:RBr}
\end{center}

All the statements of the theorem can be verified directly (see \cite{Kassel} or any other book on quantum groups). In Section \ref{sec:inclusion} we will see an indirect proof based on a Yetter-Drinfel$'$d module interpretation of modules over a quasi-triangular bialgebra.

If one is only interested in solutions to the YBE, the following corollary of the above theorem is sufficient:

\begin{corollary}\label{thm:R_br_local}
Given a quasi-triangular bialgebra $(H,R)$ in $\CC,$ any left module $M$ over $H$ is a braided object in $\CC,$ with the braiding $\sigma_M = c^R_{M,M},$ which is invertible if the R-matrix $R$ is.
\end{corollary}

\begin{proof}
Apply Lemma \ref{thm:gl_br_is_loc} to the braided category structure from Theorem \ref{thm:R_br_global}.
\end{proof}

\subsection{Yetter-Drinfel$'$d modules}\label{sec:YD}

Yetter-Drinfel$'$d modules are known to be at the origin of a very vast family of solutions to the Yang-Baxter equation. According to \cite{FRT2}, \cite{FRT1} and \cite{Rad2}, this family is complete if one restricts oneself to finite-dimensional solutions over a field $\k.$ This led L.A. Lambe and D.E. Radford to use the eloquent term \textit{quantum Yang-Baxter module} instead of the more historical term \textit{Yetter-Drinfel$'$d module}, cf. \cite{Rad}. We recall here the definition of this structure and its most important properties.

\begin{definition}\label{def:YDMod}
 A \emph{Yetter-Drinfel$'$d (= YD) module} structure over a bialgebra $H$ in a symmetric category $\CC$ consists of a left $H$-module structure $\lambda$ and a right $H$-comodule structure $\delta$ on an object $M,$ satisfying the \emph{Yetter-Drinfel$'$d compatibility condition} (cf. Fig. \ref{pic:YDCondition})
\begin{equation}\label{eqn:YD}\tag{YD}
(\Id_M \otimes \mu)\circ (\delta \otimes \Id_H) \circ c_{H,M} \circ (\Id_H \otimes  \lambda) \circ (\Delta \otimes \Id_M) = 
\end{equation}
$$ (\lambda \otimes \mu) \circ (\Id_H \otimes c_{H,M}  \otimes \Id_H) \circ (\Delta \otimes \delta).$$

The category of YD modules over a bialgebra $H$ (with, as morphisms, those which are simultaneously $H$-module and $H$-comodule morphisms) is denoted by ${_H}\!\YD^H.$
\end{definition}

\begin{center}
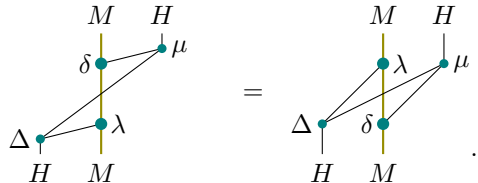

\begin{tikzpicture}[scale=0.4]
 \draw[colorM, thick] (1,4) -- (1,0);
 \draw (-1,0) -- (-1,0.5) -- (3,3.5) -- (3,4);
 \draw (1,1) -- (-1,0.5);
 \draw (1,3) -- (3,3.5);
 \node at (1,1) [right]{$\lambda$};
 \node at (1,3) [left]{$\delta$};
 \node at (3,3.5) [ right]{$\mu$};
 \node at (-1,0.5) [ left]{$\Delta$};
 \node at (1,0) [below] {$M$};
 \node at (1,4) [above] {$M$}; 
 \node at (-1,0) [below] {$H$};
 \node at (3,4) [above] {$H$};
 \fill[teal] (1,1) circle (0.2);
 \fill[teal] (1,3) circle (0.2);
 \fill[teal] (-1,0.5) circle (0.15);
 \fill[teal] (3,3.5) circle (0.15);
\node  at (6,2){$=$};
\end{tikzpicture}
\begin{tikzpicture}[scale=0.4]
\node  at (-1,1.5){};
 \draw[colorM, thick] (1,4) -- (1,0);
 \draw (-1,0) -- (-1,1) -- (3,3) -- (3,4);
 \draw (1,3) -- (-1,1);
 \draw (1,1) -- (3,3);
 \node at (1,3) [right]{$\lambda$};
 \node at (1,1) [left]{$\delta$};
 \node at (3,3) [ right]{$\mu$};
 \node at (-1,1) [ left]{$\Delta$};
 \node at (1,0) [below] {$M$};
 \node at (1,4) [above] {$M$}; 
 \node at (-1,0) [below] {$H$};
 \node at (3,4) [above] {$H$};
 \fill[teal] (1,1) circle (0.2);
 \fill[teal] (1,3) circle (0.2);
 \fill[teal] (-1,1) circle (0.15);
 \fill[teal] (3,3) circle (0.15);
 \node at (4,0) {.};
\end{tikzpicture}
   \captionof{figure}{Yetter-Drinfel$'$d compatibility condition}\label{pic:YDCondition}
\end{center}

More precisely, the definition above describes \textit{left-right Yetter-Drinfel$'$d modules}. One also encounters right-left and, if $H$ is a Hopf algebra, right-right and left-left versions. If the antipode $s$ of $H$ is invertible, then all these notions are equivalent due to the famous correspondence between left and right $H$-module structures (similarly for comodules). For example, a right action can be transformed to a left one via
\begin{center} $\lambda := \rho \circ (\Id_M \otimes s^{-1}) \circ c_{H,M}: H \otimes M \rightarrow M \; =\; $
\begin{tikzpicture}[scale=0.3]
 \draw[colorM, thick] (-1,-4) -- (-1,-2);
 \draw[rounded corners] (-1,-2.5) -- (-0.5,-3) -- (-1.5,-4);
 \node at (-1,-2.5) [left]{$\rho$};
 \node at (-0.5,-3) [right] {$s^{-1}$};
 \fill[teal] (-1,-2.5) circle (0.2);
 \fill[brown] (-0.6,-3) circle (0.15);
 \draw (-0.6,-3) circle (0.15);
\end{tikzpicture}.
\end{center}

\begin{example}\label{ex:G_as_YD}
A simple but sufficiently insightful example is given by the group algebra $H = \k G$ of a finite group $G,$ which is a Hopf algebra via the linearization of the maps $\Delta (g)= g \otimes g,$ $\varepsilon (g) = 1$ and $s(g)= g^{-1}$ for all $g \in G.$ For such an $H,$ the notion of left $H$-module (in the category $\Vect$) is easily seen to reduce to that of a $\k$-linear representation of $G,$ and the notion of right $H$-comodule to that of a $G$-graded vector space $M = \bigoplus_{g \in G} M_g,$ with 
$$\delta (m) = m \otimes g \qquad \forall m \in M_g.$$ 
The compatibility condition \eqref{eqn:YD} reads in this setting 
$$g \cdot M_h \subseteq M_{gh g^{-1}} \qquad \forall g,h \in G$$
(here the left $H$-action on $M$ is denoted by a dot). In particular, $H$ becomes a YD module over itself when endowed with the $G$-grading $H = \bigoplus_{g \in G} H_g,$ $H_g:= \k g$ and the adjoint $G$-action $$g \cdot h:= gh g^{-1}.$$ 
\end{example}

The category ${_H}\! \YD ^H$ can be endowed with a monoidal structure in several ways (cf. \cite{Rad}). We choose here the structure which makes the forgetful functor
\begin{align*}
\operatorname{For}: {_H}\! \YD ^H & \longrightarrow {_H}\!\Mod,\\
(M,\lambda, \delta) & \longmapsto (M,\lambda)
\end{align*}
monoidal, where $_H\!\Mod$ is endowed with the monoidal structure described in Section \ref{sec:R}.
Concretely, the tensor product $M \otimes N$ of YD modules $(M,\lambda_M,\delta_M)$ and $(N,\lambda_N,\delta_N)$ is endowed with the $H$-module structure \eqref{eqn:tensor_mod} and the $H$-comodule structure
\begin{equation}\label{eqn:tensor_comod}
\delta_{M \otimes N} := (\Id_{M \otimes N} \otimes \mu^{op}) \circ (\Id_M \otimes c_{H,N} \otimes \Id_H)\circ (\delta_M \otimes \delta_N),
\end{equation}
and the unit object $\II$ is endowed with the $H$-modules structure \eqref{eqn:unit_mod} and the $H$-comodule structure
\begin{equation}\label{eqn:unit_comod}
\delta_{\II} := \nu: \II \rightarrow H = \II \otimes H.
\end{equation}
Note that using the twisted multiplication $\mu^{op}$ in the definition of $\delta_{M \otimes N}$ is essential for assuring its YD compatibility with $\lambda_{M \otimes N}.$

The monoidal category ${_H}\! \YD ^H$ defined this way possesses a famous braided structure:

\begin{theorem}\label{thm:YD_br_global}
The category ${_H}\! \YD ^H$ of left-right Yetter-Drinfel$'$d modules can be endowed with the following braiding (cf. Fig. \ref{pic:YDBr}):
\begin{equation}\label{eqn:YD_br} 
c^{YD}_{M,N}:= (\Id_N \otimes \lambda_M) \circ (\delta_N \otimes \Id_M) \circ c_{M,N}.
\end{equation} 

If $H$  is moreover a Hopf algebra with the antipode $s,$ then the braiding $c^{YD}$ is invertible, with 
\begin{equation}\label{eqn:YD_br_inv}
(c^{YD}_{M,N})^{-1}:= c_{N,M} \circ  (\Id_N \otimes \lambda_M) \circ (\Id_N \otimes s \otimes \Id_M) \circ (\delta_N \otimes \Id_M).
\end{equation}
\end{theorem}

\begin{center}
\begin{tikzpicture}[scale=0.6]
 \draw[colorM, thick] (0,0) -- (1,1);
 \draw[brown, thick] (1,0) -- (0,1);
 \draw (0.85,0.85) -- (0.3,0.7);
 \node at (0.3,0.7) [left]{$\delta_N$};
 \node at (0.85,0.85) [right]{$\lambda_M$}; 
 \node at (1,0) [below] {$N$};
 \node at (0,0) [below] {$M$};
 \node at (2,0) {.};
 \fill[teal] (0.3,0.7) circle (0.1);
 \fill[teal] (0.85,0.85) circle (0.1);
\end{tikzpicture}
   \captionof{figure}{A braiding for left-right YD modules}\label{pic:YDBr}
\end{center}

The theorem can be proved by an easy direct verification. 

\pagebreak
\begin{remark}\label{rmk:YD_br_alter}
The category ${_H}\! \YD ^H$ can be endowed with a monoidal structure alternative to 
\eqref{eqn:tensor_mod}, \eqref{eqn:tensor_comod}. Namely, one can endow the tensor product of YD modules $M$ and $N$ with the ``twisted'' module and usual comodule structure:
\begin{align}
\mathring{\lambda}_{M \otimes N} &:= (\lambda_M \otimes \lambda_N) \circ (\Id_H \otimes c_{H,M} \otimes \Id_N)\circ (\Delta^{op} \otimes \Id_{M \otimes N}),\label{eqn:tensor_mod_alter}\\
\mathring{\delta}_{M \otimes N} &:= (\Id_{M \otimes N} \otimes \mu) \circ (\Id_M \otimes c_{H,N} \otimes \Id_H)\circ (\delta_M \otimes \delta_N).\label{eqn:tensor_comod_alter}
\end{align}

Theorem \ref{thm:YD_br_global} remains true in this setting if one replaces the braiding $c^{YD}$ with its alternative version
\begin{equation}\label{eqn:YD_br_alter} 
\crot^{YD}_{M,N}:= c_{M,N} \circ (\Id_M \otimes \lambda_N) \circ (\delta_M \otimes \Id_N).
\end{equation} 

This construction is best explained graphically. First, observe that the notions of bialgebra, YD module and braiding are stable by the \emph{central symmetry}. In other words, the sets of diagrams representing the axioms defining these notions are stable by an angle $\pi$ rotation (hence the notations $\crot$  etc. evoking rotation). Now, an angle $\pi$ rotation of the $H$-module structure \eqref{eqn:tensor_mod} is the $H$-comodule structure \eqref{eqn:tensor_comod_alter}, and similarly for \eqref{eqn:tensor_comod} and \eqref{eqn:tensor_mod_alter}. To conclude, note that the braiding $\crot^{YD}$ is precisely an angle $\pi$ rotation of ${c}^{YD}$ (cf. Figs \ref{pic:YDBr} and \ref{pic:YDBr_alter}). This alternative structure will be used in Section \ref{sec:YD_YB}.

\begin{center}
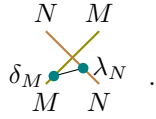

\begin{tikzpicture}[scale=0.7]
 \draw[colorM, thick] (0,0) -- (1,1);
 \draw[brown, thick] (1,0) -- (0,1);
 \draw (0.15,0.15) -- (0.7,0.3);
 \node at (0.15,0.15) [left]{$\delta_M$};
 \node at (0.7,0.3) [right]{$\lambda_N$}; 
 \node at (1,0) [below] {$N$};
 \node at (0,0) [below] {$M$};
 \node at (2,0) {.};
 \node at (0,1) [above] {$N$}; 
 \node at (1,1) [above] {$M$};
 \fill[teal] (0.15,0.15) circle (0.1);
 \fill[teal] (0.7,0.3) circle (0.1);
\end{tikzpicture}
   \captionof{figure}{An alternative braiding for left-right YD modules}\label{pic:YDBr_alter}
\end{center}

\end{remark}

If one is only interested in solutions to the YBE, the following corollary is sufficient:

\begin{corollary}\label{thm:YD_br_local}
Given a bialgebra $H$ in $\CC,$ any left-right YD module $M$ over $H$ is a braided object in $\CC,$ with the braiding $\sigma_M = c^{YD}_{M,M},$ which is invertible if $H$ is moreover a Hopf algebra.
\end{corollary}

\begin{proof}
Apply Lemma \ref{thm:gl_br_is_loc} to the braided category structure from Theorem \ref{thm:YD_br_global}.
\end{proof}

\begin{example}\label{ex:G_as_YD_br}
Applied to Example \ref{ex:G_as_YD}, the corollary gives an invertible braiding 
\begin{align*}
c^{YD}_{H,H} ( h \otimes g) &= g \otimes gh g^{-1} & \forall\: g,h \in G
\end{align*} 
for $H=\k G.$ One recognizes (the linearization of) a familiar braiding for groups, which can alternatively be obtained using the machinery of self-distributive structures.
\end{example}

\subsection{A category inclusion}\label{sec:inclusion}

The braided categories constructed in the two previous sections exhibit apparent similarities. We explain them here by interpreting the category $_H\!\Mod$ of left modules over a quasi-triangular bialgebra $H$ in a symmetric category $\CC$ as a full braided subcategory of the category ${_H}\! \YD ^H$ of left-right Yetter-Drinfel$'$d modules over $H.$ We further introduce a weaker notion of R-matrix for which the above category inclusion still holds true, in general without respecting the monoidal structures. In particular, if one is interested only in constructing solutions to the Yang-Baxter equation (cf. Corollaries \ref{thm:R_br_local} and \ref{thm:YD_br_local}), this weaker notion suffices.

\medskip

Let $(\mu,\nu)$ and $(\Delta,\varepsilon)$ be a UAA and, respectively, a coUAA structures on an object $H$ of $\CC.$ A preliminary remark is first due.

\begin{remark}\label{rmk:not_bialg}
The definition of Yetter-Drinfel$'$d module actually requires (not necessarily compatible) UAA and coUAA structures on $H$ only. The category ${_H}\! \YD ^H$ is no longer monoidal in this setting. However, a direct verification shows that Corollary \ref{thm:YD_br_local} still holds true (the argument closely repeats that from the proof of Theorem \ref{thm:YD_braided} , point 2).
\end{remark}

We thus do not suppose $H$ to be a bialgebra unless explicitely specified.

Take a left module $(M,\lambda)$ over the algebra $H$ and a morphism 
$R: \II \rightarrow H \otimes H.$
Put, as on Fig. \ref{pic:RActionGivesCoaction},
$$ \delta^R:=c_{H,M} \circ (\Id_H \otimes \lambda ) \circ (R \otimes \Id_M)  :M \rightarrow M \otimes H.$$
\begin{center}
\begin{tikzpicture}[scale=0.5]
 \node at (-3,1.5) {$ \delta^R:=$};
 \draw [colorM, thick] (0,0)  -- (0,3);
 \draw [rounded corners] (-1,1)  -- (1,3);
 \draw (-1,1) -- (0,1.5);
 \fill [teal]  (0,1.5) circle (0.2);
 \fill [palegreen] (-1,1) circle (0.15);
 \draw (-1,1) circle (0.15);
 \node at (-1,1) [below left] {$R$};
 \node at (0,1.5) [right] {$\lambda$};
\end{tikzpicture}
   \captionof{figure}{Module $+$  R-matrix $\longmapsto$ comodule}\label{pic:RActionGivesCoaction}
\end{center}

Now try to determine conditions on $R$ which make $(M,\lambda,\delta^R)$ a left-right Yetter-Drinfel$'$d module for any $M.$ One arrives to the following set of axioms:

\begin{definition}\label{def:weak_R}
 A morphism $R: \II \rightarrow H \otimes H$ is called a \emph{weak R-matrix} 
 for a UAA and coUAA object $(H,\mu,\nu,\Delta,\varepsilon)$ in $\CC$ if (cf. Fig. \ref{pic:WeakRMatrix})
 \begin{enumerate}
\item\label{item:Delta} $(\Delta \otimes \Id_H) \circ R = (\Id_{H\otimes H} \otimes \mu) \circ c^2 \circ (R \otimes R),$
\item\label{item:eps} $(\varepsilon \otimes \Id_H ) \circ R = \nu,$
\item $\mu_{H\otimes H} \circ (R \otimes \Delta) = \mu_{H\otimes H} \circ (\Delta^{op} \otimes R).$
\end{enumerate}
\end{definition}

\begin{center}
\begin{tikzpicture}[scale=0.5]
 \node at (-2,1)  {};
 \draw [rounded corners] (1,2) -- (1,1) -- (0,0);
 \draw  (-1,1) -- (0,0);
 \draw  [rounded corners] (-1,1) -- (-0.5,1.5) -- (-0.5,2);
 \draw  [rounded corners] (-1,1) -- (-1.5,1.5) -- (-1.5,2);
 \node at (2,1)  {$=$};
 \fill [teal]  (-1,1) circle (0.15);
 \fill [palegreen] (0,0) circle (0.15);
 \draw (0,0) circle (0.15);
 \node at (0,0) [below] {$R$};
 \node at (-1,1) [below] {$\Delta$};
\end{tikzpicture}
\begin{tikzpicture}[scale=0.5]
 \draw [rounded corners] (-1,2) -- (-1,1) -- (0,0);
 \draw [rounded corners] (-2,2) -- (-2,1) -- (-1,0);
 \draw  (1,2) -- (1,1) --(0,0);
 \draw  (1,1) -- (-1,0);
 \node at (2.5,0)  {,};
 \fill [teal]  (1,1) circle (0.15);
 \fill [palegreen] (0,0) circle (0.15);
 \draw (0,0) circle (0.15);
 \node at (0,0) [below] {$R$};
 \fill [palegreen] (-1,0) circle (0.15);
 \draw (-1,0) circle (0.15);
 \node at (-1,0) [below] {$R$};
 \node at (1,1) [below] {$\mu$};
\end{tikzpicture}
\begin{tikzpicture}[scale=0.5]
 \node at (-2.5,1)  {};
 \draw [rounded corners] (1,2) -- (1,1) -- (0,0);
 \draw [rounded corners] (-0.5,1) -- (-0.5,0.5) --  (0,0);
 \node at (2,1)  {$=$};
 \draw  (3,1) -- (3,2);
 \fill [orange]  (-0.5,1) circle (0.15);
 \fill [palegreen] (0,0) circle (0.15);
 \draw (0,0) circle (0.15);
 \node at (0,0) [below] {$R$};
 \node at (-0.5,1) [above] {$\varepsilon$};
 \fill [orange]  (3,1) circle (0.15);
 \node at (3,1) [right] {$\nu$};
 \node at (4,0)  {,};
\end{tikzpicture}
\begin{tikzpicture}[scale=0.6]
 \node at (-2,1)  {};
 \draw [rounded corners] (2,2) -- (2,1.3) -- (2,0.5) -- (1.5,0) -- (1.5,-0.5);
 \draw [rounded corners] (1,2) -- (1,1.3) -- (1,0.5) -- (1.5,0);
 \draw  (0,1) -- (1,1.5);
 \draw  (0,1) -- (2,1.5);
 \node at (3.5,1)  {$=$};
 \fill [teal]  (1.5,0) circle (0.15);
 \fill [teal]  (1,1.5) circle (0.15);
 \fill [teal]  (2,1.5) circle (0.15);
 \fill [palegreen] (0,1) circle (0.15);
 \draw (0,1) circle (0.15);
 \node at (0,1) [below] {$R$};
 \node at (1.5,0) [left] {$\Delta$};
 \node at (1,1.5) [above right] {$\mu$};
 \node at (2,1.5) [above right] {$\mu$};
\end{tikzpicture}
\begin{tikzpicture}[scale=0.6]
 \node at (0.5,1)  {};
 \draw [rounded corners] (2,2) -- (2,1.3) -- (1,0.5) -- (1.5,0) -- (1.5,-0.5);
 \draw [rounded corners] (1,2) -- (1,1.3) -- (2,0.5) -- (1.5,0);
 \draw  (3,1) -- (1,1.5);
 \draw  (3,1) -- (2,1.5);
 \node at (3.5,0)  {.};
 \fill [teal]  (1.5,0) circle (0.15);
 \fill [teal]  (1,1.5) circle (0.15);
 \fill [teal]  (2,1.5) circle (0.15);
 \fill [palegreen] (3,1) circle (0.15);
 \draw (3,1) circle (0.15);
 \node at (3,1) [below] {$R$};
 \node at (1.5,0) [left] {$\Delta$};
 \node at (1,1.5) [above left] {$\mu$};
 \node at (2,1.5) [above left] {$\mu$};
\end{tikzpicture}
   \captionof{figure}{Axioms for a weak R-matrix}\label{pic:WeakRMatrix}
\end{center}

One can informally interpret the first two conditions by saying that $R$ provides a duality between the UAA $(H,\mu,\nu)$ on the right and the coUAA $(H,\Delta,\varepsilon)$ on the left.

\begin{remark}\label{rmk:R_inv}
If the weak R-matrix is invertible, then axiom \ref{item:eps} is a consequence of \ref{item:Delta}: apply $\Id_H \otimes \varepsilon \otimes \Id_H$ to both sides, then multiply by $R^{-1}$ on the left and apply $\varepsilon \otimes \Id_H.$
\end{remark}

We also need the following notion:
\begin{definition}\label{def:strong_R}
 A \emph{strong R-matrix} is a weak R-matrix satisfying two additional axioms (Fig. \ref{pic:StrongRMatrix}):
\begin{enumerate}[label=\arabic*'.]
\item\label{item:Delta'} $(\Id_H \otimes \Delta) \circ R = (\mu^{op} \otimes \Id_{H\otimes H}) \circ c^2 \circ (R \otimes R),$
\item\label{item:eps'} $(\Id_H \otimes \varepsilon) \circ R = \nu.$
\end{enumerate}
\end{definition}

\begin{center}
\begin{tikzpicture}[scale=0.5]
 \draw [rounded corners] (-1,2) -- (-1,1) -- (0,0);
 \draw  (1,1) -- (0,0);
 \draw  [rounded corners] (1,1) -- (0.5,1.5) -- (0.5,2);
 \draw  [rounded corners] (1,1) -- (1.5,1.5) -- (1.5,2);
 \node at (3,1)  {$=$};
 \fill [teal]  (1,1) circle (0.15);
 \fill [palegreen] (0,0) circle (0.15);
 \draw (0,0) circle (0.15);
 \node at (0,0) [below] {$R$};
 \node at (1,1) [below] {$\Delta$};
\end{tikzpicture}
\begin{tikzpicture}[scale=0.5]
 \draw [rounded corners] (1,2) -- (1,1) -- (0,0);
 \draw [rounded corners] (2,2) -- (2,1) -- (1,0);
 \draw [rounded corners] (-1,2) -- (-1,1) --(0,0);
 \draw [rounded corners] (-1,1) -- (-1.3,0.75) -- (1,0);
 \node at (3.5,0)  {,};
 \fill [teal]  (-1,1) circle (0.15);
 \fill [palegreen] (0,0) circle (0.15);
 \draw (0,0) circle (0.15);
 \node at (0,0) [below] {$R$};
 \fill [palegreen] (1,0) circle (0.15);
 \draw (1,0) circle (0.15);
 \node at (1,0) [below] {$R$};
 \node at (-1,1) [below left] {$\mu$};
\end{tikzpicture}
\begin{tikzpicture}[scale=0.5]
 \node at (-2.5,1)  {};
 \draw [rounded corners] (-1,2) -- (-1,1) -- (0,0);
 \draw [rounded corners] (0.5,1) -- (0.5,0.5) --  (0,0);
 \node at (2,1)  {$=$};
 \draw  (3,1) -- (3,2);
 \fill [orange]  (0.5,1) circle (0.15);
 \fill [palegreen] (0,0) circle (0.15);
 \draw (0,0) circle (0.15);
 \node at (0,0) [below] {$R$};
 \node at (0.5,1) [above] {$\varepsilon$};
 \fill [orange]  (3,1) circle (0.15);
 \node at (3,1) [right] {$\nu$};
 \node at (3.5,0)  {.};
\end{tikzpicture}
   \captionof{figure}{Additional axioms for a strong R-matrix}\label{pic:StrongRMatrix}
\end{center}

A strong R-matrix satisfies all usual R-matrix axioms. Arguments similar to those from Remark \ref{rmk:R_inv} show that for invertible R-matrices the two notions coincide.

\medskip
As was hinted at above, a weak R-matrix for $H$ allows to upgrade a module structure over the algebra $H$ into a Yetter-Drinfel$'$d module structure:

\begin{theorem}\label{thm:RActionGivesCoaction}
Take a UAA and coUAA object $(H,\mu,\nu,\Delta,\varepsilon)$ in $\CC$ with a weak R-matrix $R.$ 
\begin{enumerate}
\item\label{item:R_to_YD} For any left $H$-module $(M,\lambda),$ the data $(M,\lambda,\delta^R)$ form a left-right YD module over $H.$ 
\item\label{item:R_to_YD_br} For any two $H$-modules (and hence YD modules) $(M,\lambda_M)$ and $(N,\lambda_N),$ the morphism $c^R_{M,N}$ from \eqref{eqn:R_br} coincides with $c^{YD}_{M,N}$ from \eqref{eqn:YD_br} and, for $N = M,$ defines a braiding for the object $M$ in $\CC.$
\item The category ${_H}\!\Mod$ can be seen as a full subcategory of ${_H}\!\YD^H$ via the inclusion
\begin{align*}
i_R: {_H}\!\Mod &\longhookrightarrow {_H}\!\YD^H,\\
(M,\lambda) &\longmapsto (M,\lambda,\delta^R).
\end{align*}
\item If $H$ is a bialgebra and $R$ is a strong R-matrix, then the functor $i_R$ is braided monoidal.
\end{enumerate}
\end{theorem}

\begin{proof}
\begin{enumerate}
\item The first two conditions from the definition of weak R-matrix guarantee that $\delta^R$ defines a coUAA comodule, while the last one implies the YD compatibility \eqref{eqn:YD}. 

\item The equality of the two morphisms follows from the choice of $\delta^R.$ The fact that $c^{YD}_{M,M}$ is a braiding was noticed in Remark \ref{rmk:not_bialg}.

\item We have seen in point \ref{item:R_to_YD} that $i_R$ is well defined on objects. Further, using the definition of $\delta^R$ and the naturality of $c,$ one checks that a morphism in $\CC$ respecting module structures necessarily respects comodule structures defined by $\delta^R.$ Thus $i_R$ is well defined, full and faithful on morphisms.

\item Let us now show that, under the additional conditions of this point, $i_R$ respects monoidal structures. Take two $H$-modules $(M,\lambda_M)$ and $(N,\lambda_N).$ The $H$-module structure $\lambda_{M \otimes N}$ on $M \otimes N$ is given by \eqref{eqn:tensor_mod}. The functor $i_R$ transforms it to a YD module over $H,$ with as $H$-comodule structure
\begin{align*}
\delta^R_{M \otimes N} &=c_{H,M \otimes N} \circ (\Id_H \otimes \lambda_{M \otimes N}) \circ (R \otimes \Id_{M \otimes N})\\
&= c_{H,M \otimes N} \circ (\Id_H \otimes \lambda_{M} \otimes \lambda_{N})\circ (\Id_{H\otimes H} \otimes c_{H,M} \otimes \Id_{N})  \\
&\qquad\qquad \circ (({\color{red}(\Id_H \otimes \Delta)\circ R}) \otimes \Id_{M \otimes N}).
\end{align*}
Now in ${_H}\!\YD^H$ the tensor product of $(M,\lambda_M,\delta^R_M)$ and $(N,\lambda_N,\delta^R_N)$ has  
an $H$-module structure given by $\lambda_{M \otimes N},$ and an $H$-comodule structure given by \eqref{eqn:tensor_comod}:
\begin{align*}
\delta_{M \otimes N} &= (\Id_{M \otimes N} \otimes \mu^{op}) \circ (\Id_M \otimes c_{H,N} \otimes \Id_H)\circ (\delta^R_M \otimes \delta^R_N)\\
&= (\Id_{M \otimes N} \otimes \mu^{op}) \circ (\Id_M \otimes c_{H,N} \otimes \Id_H) \\
&\qquad\qquad 
\circ (c_{H,M} \otimes c_{H,N})
\circ (\Id_H \otimes \lambda_M \otimes \Id_H \otimes \lambda_N)
\circ (R \otimes \Id_M \otimes R \otimes \Id_N) \\
 &= c_{H,M \otimes N} \circ (\Id_H \otimes \lambda_{M} \otimes \lambda_{N})\circ (\Id_{H\otimes H} \otimes c_{H,M} \otimes \Id_{N})  \\
&\qquad\qquad \circ (({\color{red} (\mu^{op} \otimes \Id_{H\otimes H}) \circ c^2 \circ (R \otimes R)}) \otimes \Id_{M \otimes N}).
\end{align*}
The reader is advised to draw diagrams in order to better follow these calculations. 
Now, axiom \ref{item:Delta'} from the definition of a strong R-matrix is precisely what is needed for the two YD structures on $M \otimes N$ to coincide.

A similar comparison of the standard $H$-comodule structure on the unit object $\II$ of $\CC$ with the one induced by $R$ shows that they coincide if and only if axiom \ref{item:eps'} is verified.

Point \ref{item:R_to_YD_br} shows that $i_R$ also respects braidings, allowing one to conclude. \qedhere
\end{enumerate}
\end{proof}

Note that in Point \ref{item:R_to_YD_br}, $c^{R}_{M,N} = c^{YD}_{M,N}$ is a morphism in $\CC$ and not in ${_H}\!\Mod$  in general, since the $H$-module structure on $M \otimes N$ is not even defined if $H$ is not a bialgebra.

In the proof of the theorem one clearly sees that the full set of strong R-matrix axioms is necessary only if one wants to construct braided monoidal categories, while the notion of weak R-matrix suffices if one is interested in the ``local'' structure of objects in $\CC$ (in particular, in solutions to the YBE) only. 

The relations between different structures from the above theorems can be presented in the following charts (in each of them one starts with a UAA and coUAA $H$ and a morphism $R: \II \rightarrow H \otimes H$):

\medskip
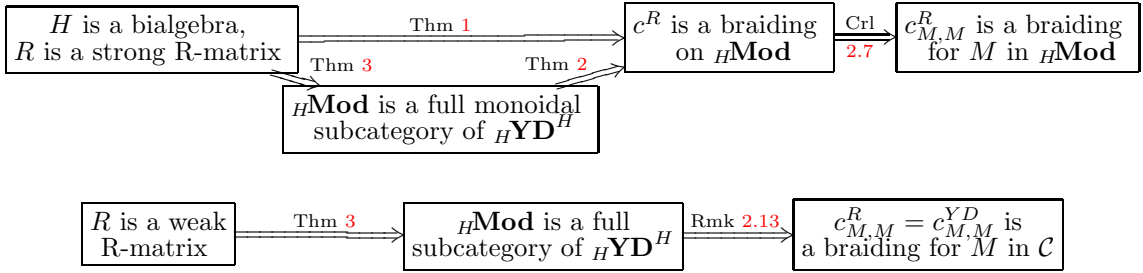
\begin{center}\label{charts}
  $ \xymatrix @!0 @R=1em @C=9pc{
    H \text{ is a bialgebra, }  &&  
    c^R \text{ is a braiding } &
    c^R_{M,M} \text{ is a braiding }\\
     R \text{ is a strong R-matrix }    
     &&  
    \text{ on } {_H}\!\Mod  &
    \text{ for } M \text{ in } {_H}\!\Mod\\
    &&&\\
     & {_H}\!\Mod \text{ is a full monoidal } &&\\
     & \text{ subcategory of } {_H}\!\YD^H &&
     \\&&&\\&&&
    \save "1,1"."2,1"!C="b1"*[F]\frm{}\restore   
    \save "4,2"."5,2"!C="b2"*[F]\frm{}\restore 
    \save "1,3"."2,3"!C="b3"*[F]\frm{}\restore   
    \save "1,4"."2,4"!C="b4"*[F]\frm{}\restore         
    \ar@{=>}^{\text{Thm \ref{thm:R_br_global}}}"b1";"b3" 
    \ar@{=>}^{\qquad \text{Thm \ref{thm:RActionGivesCoaction}}}"b1";"b2"   
    \ar@{=>}^{\text{Thm \ref{thm:YD_br_global}}\quad}"b2";"b3" 
    \ar@{=>}^{\text{Crl}\quad}_{\text{\ref{thm:R_br_local}}\quad}"b3";"b4"
}$

 $ \xymatrix @!0 @R=1em @C=12pc{
    R \text{ is a weak}  & 
    {_H}\!\Mod \text{ is a full} &
    c^R_{M,M} = c^{YD}_{M,M} \text{ is}\\
      \text{R-matrix }&
      \text{subcategory of } {_H}\!\YD^H &  
    \text{a braiding for } M \text{ in } \CC 
    \\&&\\&&
    \save "1,1"."2,1"!C="b1"*[F]\frm{}\restore   
    \save "1,2"."2,2"!C="b2"*[F]\frm{}\restore 
    \save "1,3"."2,3"!C="b3"*[F]\frm{}\restore   
    \ar@{=>}^{\text{Thm \ref{thm:RActionGivesCoaction}} \qquad}"b1";"b2" 
    \ar@{=>}^{\text{Rmk \ref{rmk:not_bialg}}}"b2";"b3" 
}$ 
\end{center} 
\medskip

Observe that if one disregards all the structural issues and limits oneself to the search of solutions to the YBE, one obtains that the only condition that should be imposed on $R$ so that \eqref{eqn:R_br} becomes a braiding is the following:
$$R_{23}*R_{13}*R_{12} = R_{12}*R_{13}*R_{23} : \II \rightarrow H \otimes H \otimes H,$$
where $ R_{12} := R \otimes \nu,$ $ R_{23} := \nu \otimes R,$ $ R_{13} :=(\Id_H \otimes c_{H,H}) \circ R_{12},$ and $*$ stands for the multiplication on $H \otimes H \otimes H$ defined by a formula analogous to \eqref{eqn:mu2}. This relation is sometimes called \emph{algebraic, or quantum, Yang-Baxter equation}. Other authors however reserve this term for \eqref{eqn:YB}.

\medskip

We finish by showing that in the Hopf algebra case, which is the most common in literature, the invertibility of a weak R-matrix is automatic (and thus a strong R-matrix is precisely the same thing as a usual R-matrix):

\begin{proposition}\label{thm:WeakRMatrixInvert}
If a {Hopf algebra} $H$ with the antipode $s$ is endowed with a weak R-matrix $R,$ then 
$$R^{-1}:=(s\otimes \Id_H ) \circ R$$
defines an inverse for $R,$ in the sense of \eqref{eqn:inv_of_R}.
\end{proposition}

\begin{proof}
Apply $(\mu \otimes \Id_H) \circ  (s \otimes \Id_H \otimes \Id_H ),$ or $(\mu \otimes \Id_H) \circ  ( \Id_H \otimes s \otimes \Id_H ),$ to both sides of the axiom \ref{item:Delta} from the definition of weak R-matrix. Axiom \ref{item:eps} and the definition \eqref{eqn:s} of the antipode allow one to conclude.
\end{proof}

\section{Yetter-Drinfel$'$d modules and the Yang-Baxter equation: the story continued}\label{sec:YD_YB}

\subsection{Braided systems}\label{sec:br_systems}

In order to describe further connections between Yetter-Drinfel$'$d modules and the Yang-Baxter equation, the following notion from \cite{Lebed2} will be useful:

\begin{definition}
\begin{itemize}
\item 
A \emph{braided system} in a monoidal category $\CC$ is an ordered finite family  $V_1,V_2,\ldots, V_r$ of objects in $\CC$ endowed with a \emph{braiding}, i.e. morphisms
 $\sigma_{i,j}:V_i \otimes V_j \rightarrow V_j \otimes V_i$ for $1\le \underline{i  \le j} \le r$  satisfying the \emph{colored Yang-Baxter equation} (= cYBE)
 \begin{equation}\label{eqn:cYB}\tag{cYBE}
 (\sigma_{j,k}\otimes \Id_i)\circ(\Id_j \otimes \sigma_{i,k})\circ(\sigma_{i,j}\otimes \Id_k) =(\Id_k \otimes \sigma_{i,j})\circ(\sigma_{i,k}\otimes \Id_j)\circ(\Id_i \otimes \sigma_{j,k})
\end{equation}
on all the tensor products $V_i \otimes V_j \otimes V_k$ with $1\le \underline{i \le j \le k} \le r.$ (Here $\Id_t$ stands for $\Id_{V_t},$ $\; 1\le t\le r.$) Such a system is denoted by $((V_i)_{1\le i\le r};(\sigma_{i,j})_{1\le i\le j\le r})$ or briefly $(\oV,\osigma).$
\item The \emph{rank} of a braided system is the number $r$ of its components.
\item A \emph{braided morphism} $\of:(\oV,\osigma) \rightarrow (\oW,\oxi)$ between two braided systems in $\CC$ of the same rank $r$ is a collection of morphisms $(f_i \in \Hom_{\CC}(V_i,W_i))_{1\le i \le r}$  respecting the braiding, i.e.
\begin{equation}\label{eqn:BrMor}
(f_j \otimes f_i) \circ \sigma_{i,j} = \xi_{i,j} \circ (f_i \otimes f_j) \quad \forall \; 1\le \underline{i  \le j} \le r.
\end{equation}
\item The category of rank $r$ braided systems and braided morphisms in $\CC$ is denoted by $\BrSyst_r(\CC).$
\end{itemize}
\end{definition}

This notion is a \emph{partial} generalization of that of braided object in $\CC,$ in the sense that a braiding is defined only on certain couples of objects (which is \underline{underlined} in the definition).

Graphically, the $\sigma_{i,j}$ component of a braiding is depicted on Fig. \ref{pic:BrSystem}.
According to the definition, one allows a strand to overcross only the strands colored with a smaller or equal index $i\in \{1,2,\ldots,r\}.$

\begin{center}
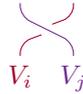

\begin{tikzpicture}[scale=0.7]
\draw [colori, rounded corners](0,0)--(0,0.25)--(0.4,0.4);
\draw [colori, rounded corners](0.6,0.6)--(1,0.75)--(1,1);
\draw [colorj, rounded corners](1,0)--(1,0.25)--(0,0.75)--(0,1);
\node  at (0,0) [colori, below] {$V_i$};
\node  at (1,0) [colorj, below] {$V_j$};
\end{tikzpicture}
\captionof{figure}{A braiding component}\label{pic:BrSystem}
\end{center}

Observe that each component of a braided system is a braided object in $\CC.$ 

Various examples of braided systems coming from algebraic considerations are presented in \cite{Lebed2}. Different aspects of those systems are studied in detail there, including their representation and homology theories, generalizing usual representation and homology theories for basic algebraic structures. 

In the next section we will describe a rank $3$ braided system constructed out of any Yetter-Drinfel$'$d module over a finite-dimensional $\k$-linear bialgebra. The braiding on this system captures all parts of the YD module structure. Several components of this braiding are inspired by the braiding $\crot^{YD}$ from \eqref{eqn:YD_br_alter}.

\subsection{Yetter-Drinfel$'$d module $\longmapsto$ braided system}\label{sec:YD_br_systems}

In this section we work in the category $\Vect$ of $\k$-vector spaces and $\k$-linear maps, endowed with the usual tensor product over $\k,$ with the unit $\k$ and with the \emph{flip}  
\begin{align}\label{eqn:flip}
c(v \otimes w) &= w \otimes v & \forall v \in V, w \in W
\end{align}
as symmetric braiding. Note however that one could stay in the general setting of a symmetric category and replace the property ``finite-dimensional'' with ``admitting a dual'' in what follow.

Recall the classical \emph{Sweedler's notation} without summation sign for comultiplications and coactions in $\Vect,$ used here and further in the paper:
\begin{align}
\Delta(h)&:=h_{(1)}\otimes h_{(2)} \in H\otimes H & \forall h \in H, \label{eqn:Sweedler}\\
\delta(m)&:=m_{(0)}\otimes m_{(1)} \in M\otimes H & \forall m \in M. \label{eqn:Sweedler'}
\end{align}

We now describe by giving explicit formulas a braided system constructed out of an arbitrary YD module. In Section \ref{sec:YD_systems_ex} this construction will be explained from a more conceptual viewpoint.

\begin{theorem}\label{thm:YDsystem_for_YDmod}
Let $H$ be a \underline{finite-dimensional} $\k$-linear bialgebra, and let $M$ be a Yetter-Drinfel$'$d module over $H.$ 
Then the rank $3$ system $(H,M,H^*)$ can be endowed with the following braiding:
\begin{align*}
\sigma_{H,H} &: h_1 \otimes h_2 \mapsto h_1 h_2 \otimes 1, 
& \sigma_{M,M} &: m_1 \otimes m_2 \mapsto m_1 \otimes m_2,\\
\sigma_{H^*,H^*} &: l_1 \otimes l_2 \mapsto \varepsilon \otimes l_1 l_2,
& \sigma_{H,M} &: h \otimes m  \mapsto h_{(2)} m \otimes h_{(1)},\\
\sigma_{H,H^*} &: h \otimes l  \mapsto l_{(1)}(h_{(2)}) \: l_{(2)} \otimes h_{(1)},
& \sigma_{M,H^*} &: m \otimes l  \mapsto l_{(1)}(m_{(1)}) \: l_{(2)} \otimes m_{(0)}.
\end{align*}
Here the signs for multiplication morphisms in $H$ and $H^*,$ as well as for the $H$-action on $M,$ are omitted for simplicity. Further, $\one$ denotes the unit of $H,$ i.e. $\nu(\alpha)= \alpha \one$ for all $\alpha \in \k.$
\end{theorem}

Multiplication, comultiplication and other structures on $H^*$ used in the theorem are obtained from those on $H$ by \emph{duality}. For instance,
\begin{align}
l_1 l_2 (h):=\Delta^* (l_1 \otimes l_2) (h) &:= l_1(h_{(2)})l_2(h_{(1)}) & \forall l_1, l_2 \in H^*, h \in H,\label{eqn:Delta_dual}\\
\one_{H^*}(h):=(\varepsilon_H)^* (1) (h)&:= \varepsilon_H(h) & \forall  h \in H.\label{eqn:eps_dual}
\end{align}
See \cite{Lebed2} for some comments on an alternative definition of the duality between $H \otimes H$ and $H^* \otimes H^*,$ which gives a slightly different formula for $\Delta^*,$ more common in literature.

The multiplication $\Delta^*$ on $H^*$ is graphically depicted on Fig. \ref{pic:DualViaRainbow}. Here and afterwards dashed lines stand for $H^*,$ and $ev$ denotes one of the \emph{evaluation} maps
\begin{align*}
 ev :\qquad  H^*\otimes H & \longrightarrow \k, & \text{or} \qquad\qquad H\otimes H^* & \longrightarrow \k,\\
 l\otimes h & \longmapsto l(h); & h\otimes l & \longmapsto l(h).
\end{align*}

\begin{center}
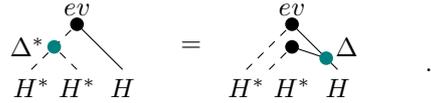

\begin{tikzpicture}[scale=0.6]
 \draw [dashed] (0,0) -- (1,1);
 \draw [dashed] (1,0) -- (0.5,0.5);
 \draw (1,1) -- (2,0);
 \node at (0,0) [below] {$H^*$};
 \node at (1,0) [below] {$H^*$};
 \node at (2,0) [below] {$H$};
 \node at (1,1) [above] {$ev$};
 \fill (1,1) circle (0.15);
 \fill [teal] (0.5,0.5) circle (0.15);
 \node at (0.5,0.5) [left] {$\Delta^*$};
 \node at (3.5,0.5) {$=$};
\end{tikzpicture}
\begin{tikzpicture}[scale=0.6]
 \draw [dashed] (0,0) -- (1,1);
 \draw [dashed] (0.5,0) -- (1,0.5);
 \draw (1,1) -- (2,0);
 \draw (1,0.5) -- (1.75,0.25);
 \node at (0,0) [below] {$H^*$};
 \node at (1,0) [below] {$H^*$};
 \node at (2,0) [below] {$H$};
 \node at (4,0) {.};
 \node at (1,1) [above] {$ev$};
 \fill (1,1) circle (0.15);
 \fill (1,0.5) circle (0.15);
 \fill [teal] (1.75,0.25) circle (0.15);
 \node at (1.75,0.5) [right] {$\Delta$};
\end{tikzpicture}
   \captionof{figure}{Dual structures on $H^*$ via the ``rainbow'' duality}\label{pic:DualViaRainbow}
\end{center}

All components of the braiding from the theorem are presented on Fig. \ref{pic:sigmaHMH*}.

\begin{center}
\begin{tikzpicture}[scale=1]
 \node at (-1,0.5) {$\sigma_{H,H} = $};
 \draw (0,1) -- (1,0);
 \draw (0,0) -- (0.5,0.5);
 \draw (0.7,0.7) -- (1,1);
 \node at (0.5,0.5) [left]{$\mu$};
 \node at (0.7,0.7) [right] {$\nu$}; 
 \fill[orange] (0.7,0.7) circle (0.075);
 \fill[teal] (0.5,0.5) circle (0.075);
 \node at (1.5,0) { $\qquad$ , $\qquad$ };
\end{tikzpicture}
\begin{tikzpicture}[scale=1]
 \node at (-2,0) {};
 \node at (-1,0.5) {$\sigma_{H,H^*} = \qquad $};
 \draw [dashed,rounded corners] (1,0)  -- (1,0.3) -- (0,1);
 \draw [rounded corners] (0,0) -- (0,0.3) -- (1,1);
 \draw [dashed,rounded corners] (0.5,0.45) -- (1,0.3) -- (1,0);
 \draw [rounded corners](0,0) -- (0,0.3) -- (0.5,0.45);
 \node at (0.5,0.45) [below] {$ev$};
 \fill (0.5,0.45) circle (0.05);
 \fill[teal] (0,0.3) circle (0.075);
 \fill[teal] (1,0.3) circle (0.075);
 \node at (0,0.3) [left] {$\Delta$};
 \node at (1,0.3) [right] {$\mu^*$};
 \node at (1.5,0) { $\qquad$ , $\qquad$ };
\end{tikzpicture}
\begin{tikzpicture}[scale=1]
 \node at (-2,0) {};
 \node at (-1,0.5) {$\sigma_{H^*,H^*} = \qquad $};
 \draw [dashed](0,0) -- (1,1);
 \draw [dashed](1,0) -- (0.5,0.5);
 \draw [dashed](0.3,0.7) -- (0,1);
 \node at (0.5,0.5) [right]{$\Delta^*$};
 \node at (0.3,0.7) [left] {$\varepsilon^*$}; 
 \fill[orange] (0.3,0.7) circle (0.075);
 \fill[teal] (0.5,0.5) circle (0.075);
 \node at (1.5,0) {,};
\end{tikzpicture}

\medskip
\begin{tikzpicture}[scale=1]
 \node at (-1,0.5) {$\sigma_{H,{M}} = \qquad $};
 \draw [colorM, thick] (1,0)  -- (0,1);
 \draw (0,0) -- (1,1);
 \draw (0.1,0.1) -- (0.7,0.3);
 \fill [teal]  (0.7,0.3) circle (0.075);
 \fill [teal]  (0.1,0.1) circle (0.05);
 \node at (0.7,0.3) [right] {$\lambda_M$};
 \node at (0.1,0.1) [above left] {$\Delta$};
 \node at (1.5,0) { $\qquad$ , $\qquad$ };
\end{tikzpicture}
\begin{tikzpicture}[scale=1]
 \node at (-2,0) {};
 \node at (-1,0.5) {$\sigma_{{M},{M}} = $};
 \draw [colorM, thick] (0,1) -- (0,0);
 \draw [colorM, thick] (1,1) -- (1,0);
 \node at (1.5,0) { $\qquad$ , $\qquad$ };
\end{tikzpicture}
\begin{tikzpicture}[scale=1]
 \node at (-1,0.5) {$\sigma_{{M},H^*} = \qquad$};
 \draw [colorM, thick] (0,0) -- (1,1);
 \draw [dashed](1,0) -- (0,1);
 \draw [dashed] (0.5,0.3) -- (0.9,0.1);
 \draw (0.1,0.1) -- (0.5,0.3);
 \node at (0.5,0.3) [below] {$ev$};
 \fill (0.5,0.3) circle (0.075); 
 \node at (0.9,0.1) [above right]{$\mu^*$};
 \node at (0.1,0.1) [above left] {$\delta_M$}; 
 \fill[teal] (0.1,0.1) circle (0.075);
 \fill[teal] (0.9,0.1) circle (0.05);
 \node at (1.5,0) {.};
\end{tikzpicture}
   
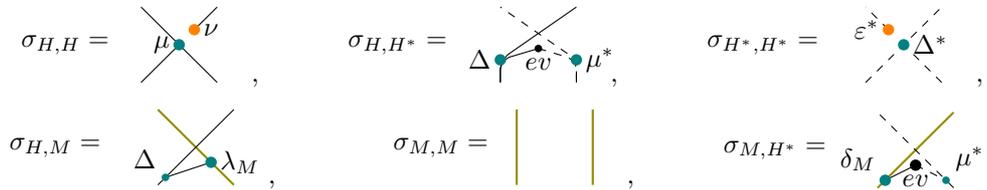
\captionof{figure}{A braiding for the system $(H,{M},H^*)$}\label{pic:sigmaHMH*}
\end{center}

In order to prove the theorem, one can simply verify by tedious but straightforward computations the $\binom{3+2}{2}=10$ instances of the cYBE involved in the definition of rank $3$ braided system. A more conceptual proof will be given in the more general setting of Sections \ref{sec:YD_systems}-\ref{sec:YD_systems_ex} (cf. Example \ref{ex:YD_as_YD_system}). In subsequent sections we will also study functoriality, precision and homology questions for this braided system.

\subsection{Yetter-Drinfel$'$d systems: definition and braided structure}\label{sec:YD_systems}

We introduce here the notion of \emph{Yetter-Drinfel$'$d system} in a \underline{symmetric} category $\CC$ and show how to endow it with a braiding. The braided system from Section \ref{sec:YD_br_systems} turns out to be a particular case of this general construction. In Section \ref{sec:YD_systems_ex} we will consider other particular cases, namely a braided system encoding the bialgabra structure (cf. \cite{Lebed2}) and a braided system of several Yetter-Drinfel$'$d modules over a common bialgebra $H.$ The latter will lead to a ``braided'' interpretation of the monoidal structures on the category ${_H}\! \YD ^H$ (Section \ref{sec:YD_tensor}).

\medskip

Let an object $H$ in $\CC$ be endowed with a UAA structure $(\mu_H,\nu_H)$ and a coUAA structure $(\Delta_H,\varepsilon_H),$ a priori not compatible. 

\begin{definition}\label{def:YDAlg}
\begin{itemize}
\item A \emph{(left-right) Yetter-Drinfel$'$d module algebra} over $H$ is the datum of a UAA structure $(\mu,\nu)$ and a YD structure $(\lambda,\delta)$ on an object $V$ in $\CC,$ such that $\mu$ and $\nu$ are morphisms of YD modules (see \eqref{eqn:tensor_mod_alter}, \eqref{eqn:tensor_comod_alter} and \eqref{eqn:unit_mod}, \eqref{eqn:unit_comod} for the $H$-(co)module structure of $V \otimes V$ and of $\II$):
\begin{align}
\delta \circ \mu &= (\mu \otimes \mu_H) \circ (\Id_V \otimes c_{H,V} \otimes \Id_H) \circ (\delta \otimes \delta), \label{eqn:YDAlg}\\
\lambda \circ (\Id_H \otimes \mu) &= \mu \circ (\lambda \otimes \lambda) \circ (\Id_H \otimes c_{H,V} \otimes \Id_V) \circ (\Delta_H^{op} \otimes \Id_{V \otimes V}), \label{eqn:YDAlg'}\\
\delta \circ \nu &= \nu \otimes \nu_H, \label{eqn:YDAlg''}\\
\lambda \circ (\Id_H \otimes \nu) &= \varepsilon_H \otimes \nu. \label{eqn:YDAlg'''}
\end{align}

\item The category of YD module algebras over $H$ in $\CC$ (with as morphisms those which are simultaneously UAA and YD module morphisms) is denoted by ${_H}\! \YDAlg ^H.$

\item Omitting the comodule structure $\delta$ from the definition, one gets the notion of \emph{$H$-module algebra}. \emph{$H$-comodule algebras} are defined similarly.
\end{itemize}
\end{definition}

The name \emph{$H^{cop}$-module algebra} would be more appropriate than \emph{$H$-module algebra} since we use the twisted compatibility condition \eqref{eqn:YDAlg'}. However we opt for the simpler notation since it causes no confusion in what follows.

For the reader's convenience, we give on Fig. \ref{pic:CompatYD} a graphical form of the compatibility relations \eqref{eqn:YDAlg}-\eqref{eqn:YDAlg'''}.
 Here and afterwards we use thick colored lines for the module $V$, and thin black lines for $H.$

\begin{center}
\begin{tikzpicture}[scale=0.5]
 \draw[colorM, thick] (0,1) -- (1,2);
 \draw[colorM, thick] (2,1) -- (1,2);
 \draw[colorM, thick] (1,4) -- (1,2);
 \draw (1,3) -- (2,4);
 \node at (1,2) [left]{$\mu$};
 \node at (1,3) [left]{$\delta$};
 \node at (2,1) [below] {$V$};
 \node at (0,1) [below] {$V$};
 \node at (1,4) [above] {$V$}; 
 \node at (2,4) [above] {$H$};
 \fill[teal] (1,2) circle (0.2);
 \fill[teal] (1,3) circle (0.2);
\node  at (3,2.5){$=$};
\end{tikzpicture}
\begin{tikzpicture}[scale=0.5]
 \draw[colorM, thick] (0,0) -- (1,2);
 \draw[colorM, thick] (2,0) -- (1,2);
 \draw[colorM, thick] (1,3) -- (1,2);
 \draw (0.5,1) -- (2,2);
 \draw (1.5,1) -- (2,2);
 \draw (2,2) -- (2,3);  
 \node at (1,2) [above left]{$\mu$};
 \node at (2,2) [above]{$\mu_H$};
 \node at (0.5,1) [left]{$\delta$};
 \node at (1.5,1) [right]{$\delta$};
 \node at (2,0) [below] {$V$};
 \node at (0,0) [below] {$V$};
 \node at (1,3) [above] {$V$}; 
 \node at (2,3) [above] {$H$};
 \fill[teal] (1,2) circle (0.2);
 \fill[teal] (2,2) circle (0.15);
 \fill[teal] (0.5,1) circle (0.2);
 \fill[teal] (1.5,1) circle (0.2);
 \draw[red,dotted] (3,-0.5) -- (3,3.5);
\end{tikzpicture}
\begin{tikzpicture}[scale=0.5]
 \draw[colorM, thick] (0,1) -- (1,2);
 \draw[colorM, thick] (2,1) -- (1,2);
 \draw[colorM, thick] (1,4) -- (1,2);
 \draw (1,3) -- (-1,1);
 \node at (1,2) [right]{$\mu$};
 \node at (1,3) [right]{$\lambda$};
 \node at (2,1) [below] {$V$};
 \node at (0,1) [below] {$V$};
 \node at (1,4) [above] {$V$}; 
 \node at (-1,1) [below] {$H$};
 \fill[teal] (1,2) circle (0.2);
 \fill[teal] (1,3) circle (0.2);
 \node  at (3.5,2.5){$=$};
\end{tikzpicture}
\begin{tikzpicture}[scale=0.5]
 \draw[colorM, thick] (0,-0.5) -- (0.5,1) -- (1,2);
 \draw[colorM, thick] (2,-0.5) -- (1.5,1) -- (1,2);
 \draw[colorM, thick] (1,2.5) -- (1,2);
 \draw (0.5,1) -- (-0.5,-0.2);
 \draw (-0.5,-0.2) -- (-0.5,-0.5);
 \draw [rounded corners] (1.5,1) -- (-0.75,0.25) -- (-0.5,-0.2);
 \node at (1,2) [left]{$\mu$};
 \node at (-0.3,0) [left]{$\Delta_H$};
 \node at (0.5,1) [left]{$\lambda$};
 \node at (1.5,1) [right]{$\lambda$};
 \node at (2,-0.5) [below] {$V$};
 \node at (0,-0.5) [below] {$V$};
 \node at (1,2.5) [above] {$V$}; 
 \node at (-1,-0.5) [below] {$H$};
 \fill[teal] (1,2) circle (0.2);
 \fill[teal] (-0.5,-0.2) circle (0.15);
 \fill[teal] (0.5,1) circle (0.2);
 \fill[teal] (1.5,1) circle (0.2);
 \draw[red,dotted] (3,-1) -- (3,3);
\end{tikzpicture}
\begin{tikzpicture}[scale=0.4]
 \node  at (1,0){};
 \draw[colorM, thick] (1,4) -- (1,2);
 \draw (1,3) -- (2,4);
 \node at (1,2) [below]{$\nu$};
 \node at (1,3) [right]{$\delta$};
 \node at (1,4) [above] {$V$}; 
 \node at (2,4) [above] {$H$};
 \fill[orange] (1,2) circle (0.2);
 \fill[teal] (1,3) circle (0.2);
 \node  at (3.5,2.5){$=$};
\end{tikzpicture}
\begin{tikzpicture}[scale=0.4]
 \node  at (1,0){};
 \draw[colorM, thick] (1,4) -- (1,2);
 \draw (2,2) -- (2,4);
 \node at (1,2) [below]{$\nu$};
 \node at (2,2) [below]{$\nu_H$};
 \node at (1,4) [above] {$V$}; 
 \node at (2,4) [above] {$H$};
 \fill[orange] (1,2) circle (0.2);
 \fill[orange] (2,2) circle (0.15);
 \draw[red,dotted] (3,0.5) -- (3,5.5);
\end{tikzpicture}
\begin{tikzpicture}[scale=0.4]
 \node  at (0.5,2){};
 \draw[colorM, thick] (1,4) -- (1,2);
 \draw (1,3) -- (0,1.5);
 \node at (1,2) [below]{$\nu$};
 \node at (1,3) [left]{$\lambda$};
 \node at (1,4) [above] {$V$}; 
 \node at (0,1.5) [below] {$H$};
 \fill[orange] (1,2) circle (0.2);
 \fill[teal] (1,3) circle (0.2);
\node  at (2.5,3){$=$};
\end{tikzpicture}
\begin{tikzpicture}[scale=0.4]
 \draw[colorM, thick] (1,4) -- (1,2);
 \draw (0,1.5) -- (0,3);
 \node at (1,2) [below]{$\nu$};
 \node at (0,3) [above]{$\varepsilon_H$};
 \node at (1,4) [above] {$V$}; 
 \node at (0,1.5) [below] {$H$};
 \fill[orange] (1,2) circle (0.2);
 \fill[orange] (0,3) circle (0.15);
\end{tikzpicture}
   
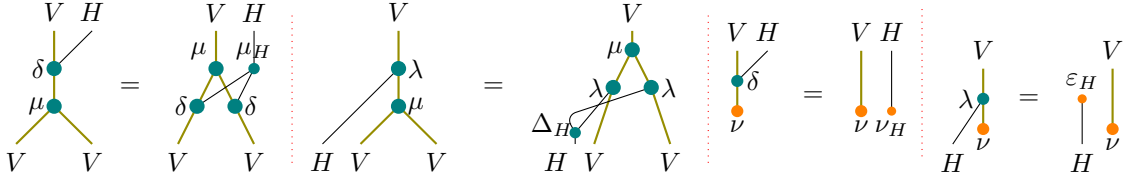
\captionof{figure}{Compatibilities between UAA and YD structures}\label{pic:CompatYD}
\end{center}

\begin{definition}\label{def:YDsystem}
A \emph{(left-right) Yetter-Drinfel$'$d system} over $H$ (= $H$-YD system) in $\CC$ is an ordered finite family $V_1,V_2,\ldots, V_r$ of objects endowed with the following structure:
\begin{enumerate}
\item $(V_1,\mu_1,\nu_1,\delta_1)$ is an $H$-comodule algebra;
\item $(V_r,\mu_r,\nu_r,\lambda_r)$ is an $H$-module algebra;
\item $(V_i,\mu_i,\nu_i,\lambda_i,\delta_i)$ is a YD module algebra over $H$ for all $1 < i < r.$ 
\end{enumerate}
\end{definition}

Now we show how to endow a YD system with a braiding:

\begin{theorem}\label{thm:YD_braided} 
A braiding can be defined on a Yetter-Drinfel$'$d system $(V_1,\ldots V_r)$ over $H$ by
\begin{align*}
 \sigma_{i,i} &:= \nu_i \otimes \mu_i \text{ or }  \mu_i \otimes \nu_i: V_i \otimes V_i  \rightarrow V_i \otimes V_i, \\
 \sigma_{i,j} &:= \crot^{YD}_{V_i,V_j}= c_{V_i,V_j} \circ (\Id_{V_i} \otimes \lambda_j) \circ (\delta_i \otimes \Id_{V_j}) : V_i \otimes V_j \rightarrow V_j \otimes V_i, \; i < j. 
\end{align*}

The $\sigma_{i,j}$ components of this braiding with $i < j$ are invertible if $H$ is a Hopf algebra.
\end{theorem}

In order to prove the theorem, the following result from \cite{Lebed2} will be useful:

\begin{proposition}\label{thm:BrSystemUAA}
Take $r$ UAAs $(V_i,\mu_i,\nu_i)_{1 \le i \le r}$ in a monoidal category $\CC$ and, for each couple of subscripts $1 \le \underline{i<j} \le r,$ take a morphism $\xi_{i,j}:V_i \otimes V_j \rightarrow V_j \otimes V_i$ natural with respect to $\nu_i$ and $\nu_j$ (in the sense of a multi-object version of  \eqref{eqn:BrAlg''}-\eqref{eqn:BrAlg'''}).
The following statements are then equivalent:
\begin{enumerate}
\item\label{item:braided} The morphisms 
\begin{align}\label{eqn:sigma_Ass}
\xi_{i,i} &:= \nu_i \otimes \mu_i  & \forall \: 1 \le i \le r
\end{align}
 complete the $\xi_{i,j}$'s into a braided system structure on $\oV.$ 
\item\label{item:mixed} Each $\xi_{i,j}$ is natural with respect to $\mu_i$ and $\mu_j$ (in the sense of a multi-object version of \eqref{eqn:BrAlg}-\eqref{eqn:BrAlg'})
and, for each triple $i<j<k,$ the $\xi_{i,j}$'s satisfy the colored Yang-Baxter equation on $V_i \otimes V_j \otimes V_k.$
\end{enumerate}
\end{proposition}

We call braided systems described in the proposition \emph{braided systems of UAAs}. In \cite{Lebed2} this structure was shown to be equivalent to that of a \emph{multi-braided tensor product of UAAs}.

The \emph{associativity braidings} $\xi_{i,i}$ from \eqref{eqn:sigma_Ass} were introduced in \cite{Lebed1}, where they were shown to encode the associativity (in the sense that the YBE for $\xi_{i,i}$ is equivalent to $\mu_i$ being associative, if one imposes that $\nu_i$ is a unit for $\mu_i$) and to capture many structural properties of the latter. We denote them by $\sigma_{Ass}(V_i).$ Observe that they are highly \emph{non-invertible} in general.

\begin{remark}\label{rmk:pre_mirror}
Some or all of the morphisms $\xi_{i,i}$ in the proposition can be replaced with their \emph{right versions} $$\sigma_{Ass}^r(V_i):=\mu_i \otimes \nu_i.$$
\end{remark}

\begin{proof}[Proof of the theorem]
First notice that the compatibility conditions \eqref{eqn:YDAlg''}-\eqref{eqn:YDAlg'''} between the units $\nu_i$ of the $V_i$'s and the $H$-(co)module structures on the $V_i$'s (cf. two last pictures on Fig. \ref{pic:CompatYD}) ensure that $\crot^{YD}_{V_i,V_j}$ is natural with respect to units. In order to deduce the theorem from Proposition \ref{thm:BrSystemUAA}, it remains to check the following two conditions:

\begin{enumerate}
\item The naturality of $\crot^{YD}_{V_i,V_j}$ with respect to $\mu_i$ and $\mu_j$ for each couple $i<j.$

Fig. \ref{pic:YDNatMuiDetails} contains a graphical proof of the naturality  with respect to $\mu_i,$ the case of $\mu_j$ being similar. Labels $V_i, V_j, H, \mu_i$ etc. are omitted for compactness.

\begin{center}
\begin{tikzpicture}[scale=0.4]
 \draw[colorM, thick,rounded corners] (0,0) -- (4,4);
 \draw[colorM, thick] (2,0) -- (1,1);
 \draw[brown!70!black, thick,rounded corners] (4,0) -- (4,1) -- (2,3) -- (2,4);
 \draw (1.5,1.5) -- (3,2);
 \fill[teal] (1,1) circle (0.2);
 \fill[teal] (1.5,1.5) circle (0.2);
 \fill[teal] (3,2) circle (0.2);
 \node  at (5,2){$\stackrel{\circled{1}}{=}$};
\end{tikzpicture}
\begin{tikzpicture}[scale=0.4]
 \draw[colorM, thick] (0,0) -- (4,4);
 \draw[colorM, thick] (2,0) -- (1.5,1.5);
 \draw[brown!70!black, thick,rounded corners] (4,0) -- (4,1) -- (2,3) -- (2,4);
 \draw[rounded corners] (11/6,0.5) -- (2.5,1.5) -- (3,2);
 \draw (0.5,0.5) -- (2.5,1.5);
 \fill[teal] (0.5,0.5) circle (0.2);
 \fill[teal] (1.5,1.5) circle (0.2);
 \fill[teal] (11/6,0.5) circle (0.2);
 \fill[teal] (2.5,1.5) circle (0.15);
 \fill[teal] (3,2) circle (0.2);
 \node  at (5,2){$\stackrel{\circled{2}}{=}$};
\end{tikzpicture}
\begin{tikzpicture}[scale=0.4]
 \draw[colorM, thick,rounded corners] (0,0) -- (4,4);
 \draw[colorM, thick] (2,0) -- (1.5,1.5);
 \draw[brown!70!black, thick,rounded corners] (4,0) -- (4,1) -- (2,3) -- (2,4);
 \draw[rounded corners] (11/6,0.5) -- (3.5,1.5);
 \draw (0.5,0.5) -- (3,2);
 \fill[teal] (0.5,0.5) circle (0.2);
 \fill[teal] (1.5,1.5) circle (0.2);
 \fill[teal] (11/6,0.5) circle (0.2);
 \fill[teal] (3.5,1.5) circle (0.2);
 \fill[teal] (3,2) circle (0.2);
 \node  at (5,2){$\stackrel{\circled{3}}{=}$};
\end{tikzpicture}
\begin{tikzpicture}[scale=0.4]
 \draw[colorM, thick,rounded corners] (0,0) -- (0,0.5) -- (3.5,4);
 \draw[colorM, thick,rounded corners] (2,0) -- (2,0.5) -- (4,3) -- (3,3.5);
 \draw[brown!70!black, thick,rounded corners] (4,0) -- (4,0.5) -- (0.5,4);
 \draw (3.6,0.9) -- (2,0.5);
 \draw (3.15,1.3) -- (0,0.5);
 \fill[teal] (3,3.5) circle (0.2);
 \fill[teal] (2,0.5) circle (0.2);
 \fill[teal] (0,0.5) circle (0.2);
 \fill[teal] (3.15,1.3) circle (0.2);
 \fill[teal] (3.6,0.9) circle (0.2);
 \node  at (5,2){$\stackrel{\circled{4}}{=}$};
\end{tikzpicture}
\begin{tikzpicture}[scale=0.4]
 \draw[colorM, thick,rounded corners] (0,0) -- (0,0.5) -- (3.5,4);
 \draw[colorM, thick,rounded corners] (2,0) -- (2,0.5) -- (4,3) -- (3,3.5);
 \draw[brown!70!black, thick,rounded corners] (4,0) -- (4,0.5) -- (0.5,4);
 \draw (3.5,1) -- (2,0.5);
 \draw (2.5,2) -- (1,1.5);
 \fill[teal] (3,3.5) circle (0.2);
 \fill[teal] (2,0.5) circle (0.2);
 \fill[teal] (1,1.5) circle (0.2);
 \fill[teal] (2.5,2) circle (0.2);
 \fill[teal] (3.5,1) circle (0.2);
 \node at (5,0.5) {.};
\end{tikzpicture}
   
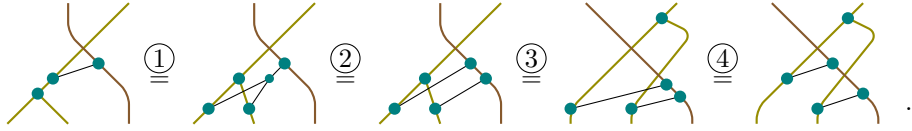
\captionof{figure}{Naturality with respect to $\mu_i$: a graphical proof}\label{pic:YDNatMuiDetails}
\end{center}

Here we use:
\begin{enumerate}[label=\protect\circled{\arabic*}]
\item the compatibility \eqref{eqn:YDAlg} between $\delta_i$ and $\mu_i$ (cf. Fig. \ref{pic:CompatYD}, picture 1);
\item the definition of $H$-module for $V_j$;
\item the naturality of the symmetric braiding $c$ of $\CC$;
\item the naturality and the symmetry \eqref{eqn:BrSymm} of $c$. 
\end{enumerate}

\item Condition \eqref{eqn:cYB} on $V_i \otimes V_j \otimes V_k$ for each triple $i<j<k.$

Due to the naturality of $c,$ this condition is equivalent to the one graphically presented on Fig. \ref{pic:YD_YB}.

\begin{center}
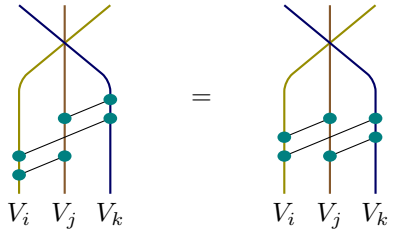

\begin{tikzpicture}[xscale=0.6,yscale=0.5]
 \draw[olive, thick,rounded corners] (0,0) -- (0,3) -- (2,5);
 \draw[brown!70!black, thick,rounded corners] (1,0) -- (1,5);
 \draw[blue!40!black, thick,rounded corners] (2,0) -- (2,3) -- (0,5);
 \draw (0,0.5) -- (1,1);
 \draw (0,1) -- (2,2);
 \draw (1,2) -- (2,2.5);
 \node at (2,0) [below] {$V_k$};
 \node at (1,0) [below] {$V_j$};
 \node at (0,0) [below] {$V_i$};
 \fill[teal] (1,2) circle (0.15);
 \fill[teal] (2,2.5) circle (0.15);
 \fill[teal] (2,2) circle (0.15);
 \fill[teal] (0,1) circle (0.15);
 \fill[teal] (1,1) circle (0.15);
 \fill[teal] (0,0.5) circle (0.15);
 \node  at (4,2.5){$=$};
\end{tikzpicture}
\begin{tikzpicture}[xscale=0.6,yscale=0.5]
 \node  at (-1,1.5){};
 \draw[olive, thick,rounded corners] (0,0) -- (0,3) -- (2,5);
 \draw[brown!70!black, thick,rounded corners] (1,0) -- (1,5);
 \draw[blue!40!black, thick,rounded corners] (2,0) -- (2,3) -- (0,5);
 \draw (0,1.5) -- (1,2);
 \draw (0,1) -- (2,2);
 \draw (1,1) -- (2,1.5);
 \node at (2,0) [below] {$V_k$};
 \node at (1,0) [below] {$V_j$};
 \node at (0,0) [below] {$V_i$};
 \fill[teal] (1,2) circle (0.15);
 \fill[teal] (2,1.5) circle (0.15);
 \fill[teal] (2,2) circle (0.15);
 \fill[teal] (0,1) circle (0.15);
 \fill[teal] (1,1) circle (0.15);
 \fill[teal] (0,1.5) circle (0.15);
 \node at (3,0) {.};
\end{tikzpicture}
   \captionof{figure}{Yang-Baxter equation for Yetter-Drinfel$'$d modules}\label{pic:YD_YB}
\end{center}
To prove this, one needs 
\begin{itemize}
\item the defining property of right $H$-comodule for $V_i,$
\item the defining property of left $H$-module for $V_k,$
\item the Yetter-Drinfel$'$d property for $V_j.$ 
\end{itemize}
\end{enumerate}

\medskip
As for the invertibility statement, if $H$ is a Hopf algebra, then the inverse of $\sigma_{i,j} = \crot^{YD}_{V_i,V_j},$ $i < j,$ can be explicitely given by the formula
\begin{align*}
 \sigma_{i,j}^{-1} &=  (\Id_{V_i} \otimes \lambda_j) \circ (\Id_{V_i} \otimes s \otimes \Id_{V_j}) \circ (\delta_i \otimes \Id_{V_j}) \circ c_{V_j,V_i} : V_j \otimes V_i \rightarrow V_i \otimes V_j. \qedhere
\end{align*}

\end{proof}

\begin{remark}\label{rmk:YDsystem_implicit_bialg}
No compatibility between the algebra and coalgebra structures on $H$ are demanded explicitly. However, other properties of a YD system dictate that $H$ should be not too far from a bialgebra, at least as far as (co)actions are concerned. This statement is made concrete and is proved (in the $H$-comodule case) on Fig. \ref{pic:AlmostBialg}; cf. Fig. \ref{pic:Bialg} for the definition of bialgebra.
\begin{center}
\begin{tikzpicture}[scale=0.4]
 \draw[colorM, thick] (0,0) -- (0,4);
 \draw[colorM, thick] (2,0) -- (0,2);
 \draw (0,1) -- (1,2) -- (1,4);
 \draw (1,1) -- (1,3) -- (2,4); 
 \node at (0,0) [below] {$V_i$};
 \node at (2,0) [below] {$V_i$};
 \node at (0,4) [above] {$V_i$}; 
 \node at (1,4) [above] {$H$};
 \node at (2,4) [above] {$H$};
 \fill[teal] (0,1) circle (0.2);
 \fill[teal] (0,2) circle (0.2);
 \fill[teal] (1,1) circle (0.2);
 \fill[teal] (1,2) circle (0.15);
 \fill[teal] (1,3) circle (0.15);
\node  at (3,2){$=$};
\end{tikzpicture}
\begin{tikzpicture}[scale=0.4]
 \draw[colorM, thick] (0,0) -- (0,4);
 \draw[colorM, thick] (1,0) -- (0,1);
 \draw (1,3) -- (1,4);
 \draw (0,2) -- (2,4); 
 \node at (0,0) [below] {$V_i$};
 \node at (1,0) [below] {$V_i$};
 \node at (0,4) [above] {$V_i$}; 
 \node at (1,4) [above] {$H$};
 \node at (2,4) [above] {$H$};
 \fill[teal] (0,1) circle (0.2);
 \fill[teal] (0,2) circle (0.2);
 \fill[teal] (1,3) circle (0.15);
\node  at (3,2){$=$};
\end{tikzpicture}
\begin{tikzpicture}[scale=0.4]
 \draw[colorM, thick] (0,0) -- (0,4);
 \draw[colorM, thick] (2,0) -- (0,1);
 \draw (0,2) -- (2,4);
 \draw (0,3) -- (1,4); 
 \node at (0,0) [below] {$V_i$};
 \node at (2,0) [below] {$V_i$};
 \node at (0,4) [above] {$V_i$}; 
 \node at (1,4) [above] {$H$};
 \node at (2,4) [above] {$H$};
 \fill[teal] (0,1) circle (0.2);
 \fill[teal] (0,2) circle (0.2);
 \fill[teal] (0,3) circle (0.2);
\node  at (3,2){$=$};
\end{tikzpicture}
\begin{tikzpicture}[scale=0.4]
 \draw[colorM, thick] (0,0) -- (0,4);
 \draw[colorM, thick] (2,0) -- (0,3);
 \draw (0,1.5) -- (1,3) -- (1,4) -- (1,1.5);
 \draw (1.5,0.75) -- (1.5,4); 
 \draw (0,0.75) -- (1.5,3); 
 \node at (0,0) [below] {$V_i$};
 \node at (2,0) [below] {$V_i$};
 \node at (0,4) [above] {$V_i$}; 
 \node at (1,4) [above] {$H$};
 \node at (2,4) [above] {$H$};
 \fill[teal] (0,1.5) circle (0.17);
 \fill[teal] (0,0.75) circle (0.17);
 \fill[teal] (1,1.5) circle (0.17);
 \fill[teal] (1,3) circle (0.12);
 \fill[teal] (0,3) circle (0.17);
 \fill[teal] (1.5,0.75) circle (0.17);
 \fill[teal] (1.5,3) circle (0.12);
 \node  at (3,2){$=$};
\end{tikzpicture}
\begin{tikzpicture}[scale=0.4]
 \draw[colorM, thick] (0,0) -- (0,4);
 \draw[colorM, thick] (2,0) -- (0,2);
 \draw (0,1) -- (2,3) -- (2,2) -- (1,3);
 \draw (1,1) -- (2,2) -- (2,4); 
 \draw (1,2) -- (1,4);
 \node at (0,0) [below] {$V_i$};
 \node at (2,0) [below] {$V_i$};
 \node at (0,4) [above] {$V_i$}; 
 \node at (1,4) [above] {$H$};
 \node at (2,4) [above] {$H$};
 \node at (3,0) [below] {.};
 \fill[teal] (0,1) circle (0.2);
 \fill[teal] (0,2) circle (0.2);
 \fill[teal] (1,1) circle (0.2);
 \fill[teal] (1,2) circle (0.15);
 \fill[teal] (1,3) circle (0.15);
 \fill[teal] (2,2) circle (0.15);
 \fill[teal] (2,3) circle (0.15);
\end{tikzpicture}
   
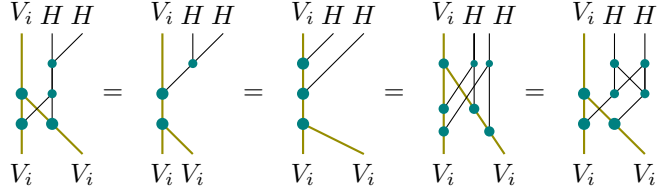
\captionof{figure}{Almost a bialgebra}\label{pic:AlmostBialg}
\end{center}
\end{remark}

\subsection{Yetter-Drinfel$'$d systems: examples}\label{sec:YD_systems_ex}

The examples of $H$-YD systems we give below contain usual YD modules over $H,$ a priori without a UAA structure. In the following lemma we explain how to overcome this lack of structure by introducing a \emph{formal unit} into a YD module:

\begin{lemma}\label{thm:YD_mod_UAA}
Let $(M,\lambda,\delta)$ be a YD module over a UAA and coUAA $H$ in a symmetric \underline{additive} category $\CC.$ This YD module structure can be extended to $\widetilde{M} := M \oplus \II$ as follows:
\begin{align*}
\widetilde{\lambda}|_{H \otimes \II} &:= \varepsilon_H, &\widetilde{\lambda}|_{H \otimes M} &:= \lambda,\\
\widetilde{\delta}|_{\II} &:= \nu_H, &\widetilde{\delta}|_{M} &:= \delta.
\end{align*}
Moreover, combined with the trivial UAA structure on $\widetilde{M}$: 
$$\mu|_{M \otimes M} = 0, \qquad\qquad \mu|_{\II \otimes \widetilde{M}} = \mu|_{\widetilde{M} \otimes \II} = \Id_{\widetilde{M}}, \qquad\qquad  \nu = \Id_{\II},$$
this extended YD module structure turns $M$ into an $H$-YD module algebra.
\end{lemma}

\begin{proof} Direct verifications. \end{proof} 

We need $\CC$ to be additive in order to ensure the existence of direct sums and zero morphisms. Our favorite category $\Vect$ is additive.

In what follows we mostly work with YD modules admitting moreover a compatible UAA structure; the lemma shows that this assumption does not reduce the generality of our constructions.

\medskip
Everything is now ready for our main example of $H$-YD systems. Recall Sweedler's notation \eqref{eqn:Sweedler}-\eqref{eqn:Sweedler'}.

\begin{proposition}\label{thm:YDsystem_ex}
Let $H$ be a finite-dimensional $\k$-linear bialgebra, 
 and let $M_1, \ldots, M_r$ be Yetter-Drinfel$'$d module algebras over $H.$ 
Consider $H$ itself as a UAA via morphisms $\mu_H$ and $\nu_H$ and as an $H$-comodule via $\Delta_H.$ Further, consider its dual space $H^*$ as a UAA via dual morphisms $(\Delta_H)^*$ and $(\varepsilon_H)^*$ and as an $H$-module via 
\begin{align}\label{eqn:H_acts_on_H*} 
h\cdot l &:= h(l_{(1)})l_{(2)} & \forall\: h \in H, l \in H^*
\end{align}
(cf. Fig. \ref{pic:HActsOnH*}).
 These structures turn the family $(H,M_1, \ldots, M_r,H^*)$ into an $H$-YD system.
\end{proposition}

\begin{center}
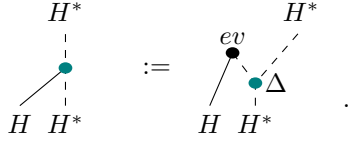

\begin{tikzpicture}[xscale=0.6,yscale=0.5]
 \node at (-2,1) {};
 \draw [dashed] (1,0) -- (1,2);
 \draw (0,0) -- (1,1);
 \fill [teal] (1,1) circle (0.15);
 \node at (1,2) [above] {$H^*$};
 \node at (1,0) [below] {$H^*$};
 \node at (0,0) [below] {$H$};
 \node at (3,1) {$:=$};
\end{tikzpicture}
\begin{tikzpicture}[xscale=0.6,yscale=0.5]
 \draw [dashed] (1,0)  -- (1,0.6) -- (2,2);
 \draw [dashed] (0.5,1.4) -- (1,0.6);
 \draw (0,0) -- (0.5,1.4);
 \node at (0.5,1.4) [above] {$ev$};
 \fill (0.5,1.4) circle (0.15);
 \node at (1,0.6) [right] {$\Delta$};
 \fill [teal] (1,0.6) circle (0.15);
 \node at (2,2) [above] {$H^*$};
 \node at (1,0) [below] {$H^*$};
 \node at (0,0) [below] {$H$};
 \node at (3,0) {.};
\end{tikzpicture}
   \captionof{figure}{The action of $H$ on $H^*$}\label{pic:HActsOnH*}
\end{center}

\begin{proof}
Almost all axioms from Definition \ref{def:YDsystem} are automatic. One should check only the following points:
\begin{itemize}
\item the $H$-comodule structure $\Delta_H$ on $H$ is compatible with the UAA structure $(\mu_H,\nu_H)$;
\item the map \eqref{eqn:H_acts_on_H*} indeed defines an $H$-action on $H^*$;
\item the $H$-module structure \eqref{eqn:H_acts_on_H*} on $H^*$ is compatible with the UAA structure $((\Delta_H)^*,(\varepsilon_H)^*)$.
\end{itemize}
These straightforward verifications use bialgebra axioms and the definition of dual morphisms. They are easily effectuated using the graphical calculus.
\end{proof}

Applying Theorem \ref{thm:YD_braided} to the $H$-YD system from Proposition \ref{thm:YDsystem_ex}, and choosing braiding component $\sigma_{Ass}$ for $H^*$ and the $M_i$'s, and its right version $\sigma_{Ass}^r$ for $H$ (cf. Remark \ref{rmk:pre_mirror}), one obtains the following braided system:

\begin{corollary}\label{thm:YDsystem_ex_crl}
In the settings of Proposition \ref{thm:YDsystem_ex}, the system $(H,M_1, \ldots, M_r,H^*)$ can be endowed with the following braiding:
\begin{align*}
\sigma_{H,H} &:= \sigma_{Ass}^r(H) = \mu_H \otimes \nu_H: & h_1 \otimes h_2 &\mapsto h_1 h_2 \otimes \one_H,\\
\sigma_{H^*,H^*} &:= \sigma_{Ass}(H^*) = (\varepsilon_H)^* \otimes (\Delta_H)^*: & l_1 \otimes l_2 &\mapsto \varepsilon_H \otimes l_1 l_2,\\
\sigma_{H,H^*} &:= \crot^{YD}_{H,H^*}: & h \otimes l  &\mapsto l_{(1)}(h_{(2)}) \: l_{(2)} \otimes h_{(1)},\\ 
\sigma_{M_i,M_i} &:= \sigma_{Ass}(M_i) = \nu_i \otimes \mu_i: & m_1 \otimes m_2 &\mapsto \one_{M_i} \otimes m_1 m_2,\\
\sigma_{M_i,M_j} &:=\crot^{YD}_{M_i,M_j}: & m \otimes n  &\mapsto  m_{(1)} n \otimes m_{(0)},\\ 
\sigma_{H,M_i} &:= \crot^{YD}_{H,M_i}: & h \otimes m  &\mapsto h_{(2)} m \otimes h_{(1)},\\ 
\sigma_{M_i,H^*} &:= \crot^{YD}_{M_i,H^*}: & m \otimes l  &\mapsto l_{(1)}(m_{(1)}) \: l_{(2)} \otimes m_{(0)}.
\end{align*}
Notations analogous to those from Theorem \ref{thm:YDsystem_for_YDmod}  are used here.
\end{corollary}

See Figs \ref{pic:sigmaHMH*} and \ref{pic:YDBr_alter} for a graphical presentation of the components of the braiding from the corollary. Notation $\osigma_{YDAlg}$ will further be used for this braiding.

The simplest particular cases of the corollary already give interesting examples:

\begin{example}\label{ex:bialg_as_YD_system}
In the extreme case $r=0$ one gets the braiding on the system $(H,H^*)$ studied in \cite{Lebed2}, where it was shown that:
\begin{itemize}
\item it \emph{encodes the bialgebra structure} (e.g. the cYBE on $H \otimes H \otimes H^*$ or on $H \otimes H^* \otimes H^*$ is equivalent to the main bialgebra axiom \eqref{eqn:cat_bialg});
\item it allows to construct a fully faithful functor
$$^*\!\bialg \longhookrightarrow ^*\!\BrSyst^{\bullet\bullet}_2,$$
where $^*\!\bialg$ is the category of finite-dimensional bialgebras and bialgebra \textit{iso}morphisms (hence notation $^*$) in $\Vect,$ and $^*\!\BrSyst^{\bullet\bullet}_2$ is the category of rank $2$ braided systems and their \textit{iso}morphisms in $\Vect,$ endowed with some additional structure (hence notation $^{\bullet\bullet}$);
\item the invertibility of the component $\sigma_{H,H^*}$ of this braiding is equivalent to the \emph{existence of the antipode} for $H$;
\item braided modules over this system (cf. \cite{Lebed2} for a definiton) are precisely \emph{Hopf modules} over $H$;
\item the braided homology theory for this system (cf. Section \ref{sec:hom}) includes the \emph{Gerstenhaber-Schack bialgebra homology}, defined in \cite{GS90}.
\end{itemize}
\end{example}

\begin{example}\label{ex:YD_as_YD_system}
The $r=1$ case gives a braiding on the system $(H,M,H^*)$ for a YD module algebra $M$ over $H.$ Now observe that one can substitute the $\sigma_{M,M} = \sigma_{Ass}(M)$ component of $\osigma_{YDAlg}$ with the trivial one, $\sigma_{M,M} = \Id_{M \otimes M},$ since the instances of the cYBE involving two copies of $M$ (which are the only instances affected by the modification of $\sigma_{M,M}$) become automatic. Moreover, with this new $\sigma_{M,M}$ the UAA structure on $M$ is no longer necessary. One thus obtains a braiding on $(H,M,H^*)$ for any $H$-YD module $M,$ which coincides with the one in Theorem \ref{thm:YDsystem_for_YDmod}. This gives an alternative proof of that theorem.
\end{example}

\begin{example}\label{ex:ModAlg_as_YD_system}
If $M$ is just a module algebra, one can extract a braided system $(H,M)$ from the construction of the previous example. Studying braided differentials for this system (cf. Section \ref{sec:hom}), one recovers the \emph{deformation bicomplex of module algebras}, introduced by D. Yau in \cite{Yau}. Note that in this example $H$ need not be finite-dimensional.
\end{example}

\begin{example}\label{ex:multi_YD_as_YD_system}
The argument from example \ref{ex:YD_as_YD_system} also gives a braiding on the system $(H,M_1, \ldots,$ %
 $M_r ,H^*)$ for $H$-YD modules $M_1, \ldots,M_r,$ a priori without UAA structures. This braiding is obtained by replacing all the $\sigma_{M_i,M_i}$'s from Corollary \ref{thm:YDsystem_ex_crl} with the trivial ones, $\sigma_{M_i,M_i} = \Id_{M_i \otimes M_i}.$ It is denoted by $\osigma_{YD}.$
\end{example}

\subsection{Yetter-Drinfel$'$d systems: properties}\label{sec:YD_systems_prop}

Now let us study several properties of the braiding from Corollary \ref{thm:YDsystem_ex_crl}, namely its \emph{functoriality} and the \emph{precision} of encoding YD module algebra axioms.

\begin{proposition}\label{thm:YDsyst_br_funct}
In the settings of Proposition \ref{thm:YDsystem_ex}, one has a faithful functor
\begin{align}
({_H}\! \YDAlg ^H)^{\times r} & \stackrel{\ibr}{\longhookrightarrow} \BrSyst_{r+2},\label{eqn:cat_incl_YD_BrSyst}\\
\overline{M} =(M_1, \ldots, M_r) & \longmapsto (H, \overline{M},H^*; \osigma_{YDAlg}),\notag\\
\of =(f_i:M_i \rightarrow N_i)_{1 \le i \le r} & \longmapsto (\Id_H, \of, \Id_{H^*}).\notag
\end{align}
Moreover, if a braided morphism from $\ibr(\overline{M})$ to $\ibr(\overline{N})$ has the form $(\Id_H, \of, \Id_{H^*})$ and if the $f_i$'s respect units (in the sense of \eqref{eqn:respect_nu}), then all the $f_i:M_i \rightarrow N_i$ are morphisms in ${_H}\! \YDAlg ^H.$
\end{proposition}

\begin{proof}
Corollary \ref{thm:YDsystem_ex_crl} says that the functor is well defined on objects. It remains to study the compatibility condition \eqref{eqn:BrMor} for the collection $(\Id_H, \of, \Id_{H^*})$ and each component of the braiding $\osigma_{YDAlg}.$
\begin{enumerate}
\item On $H \otimes H,$ $H \otimes H^*$ and $H^* \otimes H^*$ condition \eqref{eqn:BrMor} trivially holds true.
\item On $M_i \otimes M_i,$ condition \eqref{eqn:BrMor} reads
$$(f_i \otimes f_i) \circ (\nu_{M_i} \otimes \mu_{M_i})= (\nu_{N_i} \otimes \mu_{N_i}) \circ (f_i \otimes f_i),$$
which, due to \eqref{eqn:respect_nu}, becomes
$$\nu_{N_i} \otimes (f_i \circ \mu_{M_i})= (\nu_{N_i} \otimes \mu_{N_i}) \circ (f_i \otimes f_i).$$
But this is equivalent to $f_i$ respecting multiplication:
$$f_i \circ \mu_{M_i}= \mu_{N_i} \circ (f_i \otimes f_i).$$
(compose with $\mu_{N_i}$ on the left to get the less evident application), i.e., in the presence of \eqref{eqn:respect_nu}, to $f_i$ being a UAA morphism.
\item On $H \otimes M_i,$ condition \eqref{eqn:BrMor} reads
$$(f_i \otimes \Id_H) \circ c_{H,{M_i}} \circ (\Id_H \otimes \lambda_{M_i}) \circ (\Delta_H \otimes \Id_{M_i}) = c_{H,{N_i}} \circ (\Id_H \otimes \lambda_{N_i}) \circ (\Delta_H \otimes \Id_{N_i}) \circ (\Id_H \otimes f_i),$$
which simplifies as
$$(\Id_H \otimes (f_i \circ \lambda_{M_i})) \circ (\Delta_H \otimes \Id_{M_i}) = 
 (\Id_H \otimes \lambda_{N_i}) \circ (\Id_{H \otimes H} \otimes f_i) \circ (\Delta_H \otimes \Id_{M_i}).$$
 This is equivalent to $f_i$ respecting the $H$-module structure:
$$(f_i \circ \lambda_{M_i})=  \lambda_{N_i} \circ (\Id_{ H} \otimes f_i)$$
(compose with $\varepsilon \otimes \Id_{N_i}$ on the left to get the less evident application).

\item Similarly, on $M_i \otimes H^*$ condition \eqref{eqn:BrMor} is equivalent to $f_i$ respecting the $H$-comodule structure.
\item Condition \eqref{eqn:BrMor} holds true on $M_i \otimes M_j, \: i<j,$ if $f_j$ respects the $H$-module structures and $f_i$ respects the $H$-comodule structures.
\end{enumerate}

The reader is advised to draw diagrams in order to better follow the proof.

This analysis shows that $\ibr$ is well defined on morphisms. Moreover, it shows that if a braided morphism from $\ibr(\overline{M})$ to $\ibr(\overline{N})$ has the form $(\Id_H, \of, \Id_{H^*})$ and if the $f_i$'s respect units, then all the $f_i:M_i \rightarrow N_i$ are simultaneously  morphisms of UAAs (point 2 above), of $H$-modules (point 3) and of $H$-comodules (point 4), which precisely means that they are morphisms in ${_H}\! \YDAlg ^H.$

The faithfulness of the functor is tautological.
\end{proof}

\begin{remark}\label{rmk:YDsyst_br_funct_full}
The second statement of the proposition allows to call the functor $\ibr$ \emph{``essentially full''}. A precise description of $(\alpha,\of, \beta) \in \Hom_{\BrSyst_{r+2}}(\ibr(\overline{M}),\ibr(\overline{N}))$ can be obtained using the results recalled in Example \ref{ex:bialg_as_YD_system}. Namely, if one imposes some additional restrictions ($\alpha$ and $\beta$ should be invertible and respect the units, the $f_i$'s should respect units as well), then 
\begin{itemize}
\item $\alpha$ is a bialgebra automorphism, and coinsides with $(\beta^{-1})^*$;
\item each $f_i$ is a UAA morphism;
\item $\alpha$ and the $f_i$ 's are compatible with the YD structures on the $M_i$'s and $N_i$'s:
\begin{align*}
f_i \circ \lambda_{M_i} &= \lambda_{N_i} \circ (\alpha \otimes f_i): H \otimes M_i\rightarrow N_i,\\
\delta_{N_i} \circ f_i &= (f_i \otimes \alpha) \circ \delta_{M_i}: M_i\rightarrow N_i \otimes H.
\end{align*}
\end{itemize}
\end{remark}

\begin{remark}\label{rmk:YDsyst_br_funct_noUAA}
Following Example \ref{ex:multi_YD_as_YD_system}, one gets a similar ``essentially full'' and faithful functor
\begin{align*}
({_H}\! \YD ^H)^{\times r} & \stackrel{\ibr}{\longhookrightarrow} \BrSyst_{r+2},\\
\overline{M} =(M_1, \ldots, M_r) & \longmapsto (H, \overline{M},H^*; \osigma_{YD}),\\
\of =(f_i:M_i \rightarrow N_i)_{1 \le i \le r} & \longmapsto (\Id_H, \of, \Id_{H^*}).
\end{align*}
\end{remark}

We next show that each instance of the cYBE for the braiding $\osigma_{YDAlg}$ corresponds precisely to an axiom from the Definition \ref{def:YDAlg} of YD module algebra structure, modulo some minor normalization constraints. We thus obtain a ``braided'' interpretation of each of these axioms, extending the classical ``braided'' interpretation of the YD compatibility condition \eqref{eqn:YD} (Corollary \ref{thm:YD_br_local}). 

\begin{proposition}\label{thm:YDsyst_br_precision}
Take a finite-dimensional $\k$-linear bialgebra $H,$ a vector space $V$ over $\k,$ and morphisms $\mu: V \otimes V \rightarrow V,$ $\nu: \k \rightarrow V,$ $\lambda: H \otimes V \rightarrow V,$ $\delta: V \rightarrow V \otimes H,$ a priori satisfying no compatibility relations. Consider the collection of $\sigma$'s defined in Corollary \ref{thm:YDsystem_ex_crl} (for $r=1$). One has the following equivalences:

\begin{enumerate}
\item cYBE for $H \otimes V \otimes H^*$ $\Longleftrightarrow$ YD compatibility condition \eqref{eqn:YD} for $(V, \lambda, \delta).$
\item  cYBE for $H \otimes  H \otimes V$ $\Longleftrightarrow$ $\lambda$ defines an $H$-module. 

Here one supposes that $\nu_H$ acts via $\lambda$ by identity (in the sense of \eqref{eqn:Amod'}).
\item  cYBE for $V \otimes H^* \otimes H^*$  $\Longleftrightarrow$ $\delta$ defines an $H$-comodule.

Here one supposes that $\varepsilon_H$ coacts via $\delta$ by identity (in the sense dual to \eqref{eqn:Amod'}).
\item  cYBE for $H \otimes V \otimes V$ $\Longleftrightarrow$ $\lambda$ respects the multiplication $\mu$ (in the sense of \eqref{eqn:YDAlg'}).

Here one supposes $\lambda$ to respect $\nu$ (in the sense of \eqref{eqn:YDAlg'''}).
\item  cYBE for $V \otimes V \otimes H^*$ $\Longleftrightarrow$ $\delta$ respects the multiplication $\mu$ (in the sense of \eqref{eqn:YDAlg}).

Here one supposes $\delta$ to respect $\nu$ (in the sense of \eqref{eqn:YDAlg''}).
\item  cYBE for $V \otimes V \otimes V$ $\stackrel{\nu \text{ is a unit for } \mu}{\Longleftrightarrow}$ associativity of $\mu.$ 
\end{enumerate}
\end{proposition}

\begin{proof}
We prove only the first equivalence here, the other points being similar.

The cYBE for $H \otimes V \otimes H^*$ is graphically depicted on Fig. \ref{pic:cYBE_YD} a). Using the naturality of $c,$ one transforms this into the diagram on Fig. \ref{pic:cYBE_YD} b).
Applying $\varepsilon_{H^*} \otimes V \otimes \varepsilon_H$ to both sides and using the definition of $\Delta_{H^*}=\mu_H^*$ in terms of $\mu_H,$ one gets precisely \eqref{eqn:YD}.

\begin{center}
\begin{tikzpicture}[scale=0.8]
\draw [ rounded corners](0,-0.5)--(0,0.25)--(1,0.75)--(1,1.25)--(2,1.75)--(2,3);
\draw [colorM, thick, rounded corners](1,-0.5)--(1,0.25)--(0,0.75)--(0,2.25)--(1,2.75)--(1,3);
\draw [dashed, rounded corners](2,-0.5)--(2,1.25)--(1,1.75)--(1,2.25)--(0,2.75)--(0,3);
\draw (0,-0.25) -- (1,0.25);
\fill [teal] (0,-0.25) circle (0.1);
\fill [teal] (1,0.25) circle (0.12);
\node  at (0,-0.25) [left] {$\Delta_H$};
\node  at (1,0.25) [right] {$\lambda$};
\draw [rounded corners] (1,0.8) -- (1.3,1.2) -- (1.5,1.25);
\draw [dashed, rounded corners] (2,0.8)-- (1.7,1.2) -- (1.5,1.25);
\fill  (1.5,1.25) circle (0.1);
\fill [teal] (1,0.8) circle (0.1);
\fill [teal] (2,0.8) circle (0.1);
\node  at (1.2,1) [left] {$\Delta_H$};
\node  at (2,1) [right] {$\mu_H^*$};
\node  at (1.5,1.25) [below] {$ev$};
 \draw [rounded corners] (0,1.8) -- (0.3,2.2) -- (0.5,2.25);
\draw [dashed, rounded corners] (1,1.8)-- (0.7,2.2) -- (0.5,2.25);
\fill  (0.5,2.25) circle (0.1);
\fill [teal] (0,1.8) circle (0.12);
\fill [teal] (1,1.8) circle (0.1);
\node  at (0,2) [left] {$\delta$};
\node  at (1,2) [right] {$\mu_H^*$};
\node  at (0.5,2.25) [below] {$ev$};
\node  at (1,-0.5) [below] {$H \otimes V \otimes H^*$};
\node  at (3.5,1){$=$};
\node  at (3,-1.5){$a)$};
\end{tikzpicture}
\begin{tikzpicture}[scale=0.8]
\draw [colorM, thick, rounded corners](1,0.5)--(1,1.25)--(2,1.75)--(2,3.25)--(1,3.75)--(1,4);
\draw [rounded corners](0,0.5)--(0,2.25)--(1,2.75)--(1,3.25)--(2,3.75)--(2,4);
\draw [dashed, rounded corners](2,0.5)--(2,1.25)--(1,1.75)--(1,2.25)--(0,2.75)--(0,4);
\node  at (-0.5,-0.7){ };
\node  at (1,0.5) [below] {$H \otimes V \otimes H^*$};
\node  at (5,0.5){};
\draw (1,2.75) -- (2,3.25);
\fill [teal] (1,2.75) circle (0.1);
\fill [teal] (2,3.25) circle (0.12);
\node  at (1.1,2.85) [left] {$\Delta_H$};
\node  at (2,3.25) [right] {$\lambda$};
\draw [rounded corners] (1,0.8) -- (1.3,1.2) -- (1.5,1.25);
\draw [dashed, rounded corners] (2,0.8)-- (1.7,1.2) -- (1.5,1.25);
\fill  (1.5,1.25) circle (0.1);
\fill [teal] (1,0.8) circle (0.12);
\fill [teal] (2,0.8) circle (0.1);
\node  at (1.1,1) [left] {$\delta$};
\node  at (2,1) [right] {$\mu_H^*$};
\node  at (1.5,1.25) [below] {$ev$};
\draw [rounded corners] (0,1.8) -- (0.3,2.2) -- (0.5,2.25);
\draw [dashed, rounded corners] (1,1.8)-- (0.7,2.2) -- (0.5,2.25);
\fill  (0.5,2.25) circle (0.1);
\fill [teal] (0,1.8) circle (0.1);
\fill [teal] (1,1.8) circle (0.1);
\node  at (0,2) [left] {$\Delta_H$};
\node  at (1,2) [right] {$\mu_H^*$};
\node  at (0.5,2.25) [below] {$ev$};
\end{tikzpicture}
\begin{tikzpicture}[xscale=0.6, yscale = 0.5]
 \draw[rounded corners] (0,0) -- (0,3) -- (2,5);
 \draw[colorM, thick,rounded corners] (1,0) -- (1,5);
 \draw[dashed, rounded corners] (2,0) -- (2,3) -- (0,5);
 \draw (0,0.5) -- (1,1);
 \fill[teal] (1,1) circle (0.12);
 \fill[teal] (0,0.5) circle (0.1); 
 \draw [rounded corners] (0,1.3) -- (0.6,1.6) -- (1.3,1.75);
 \draw [dashed, rounded corners] (2,1.3) -- (1.3,1.75);
 \fill  (1.3,1.75) circle (0.1);
 \fill [teal] (0,1.3) circle (0.1);
 \fill [teal] (2,1.3) circle (0.1);
 \draw [rounded corners] (1,2.3) -- (1.3,2.7) -- (1.5,2.75);
 \draw [dashed, rounded corners] (2,2.3)-- (1.7,2.7) -- (1.5,2.75);
 \fill  (1.5,2.75) circle (0.1);
 \fill [teal] (1,2.3) circle (0.12);
 \fill [teal] (2,2.3) circle (0.1);
 \node at (1,0) [below] {$H \otimes V \otimes H^*$};
 \node  at (3.5,2.5){$=$};
 \node  at (3.5,-1.5){$b)$};
\end{tikzpicture}
\begin{tikzpicture}[xscale=0.6, yscale = 0.5]
 \draw[rounded corners] (0,0) -- (0,3) -- (2,5);
 \draw[colorM, thick,rounded corners] (1,0) -- (1,5);
 \draw[dashed, rounded corners] (2,0) -- (2,3) -- (0,5);
 \draw (0,2) -- (1,2.5);
 \fill[teal] (1,2.5) circle (0.12);
 \fill[teal] (0,2) circle (0.1); 
 \draw [rounded corners] (0,1.3) -- (0.6,1.6) -- (1.3,1.75);
 \draw [dashed, rounded corners] (2,1.3) -- (1.3,1.75);
 \fill  (1.3,1.75) circle (0.1);
 \fill [teal] (0,1.3) circle (0.1);
 \fill [teal] (2,1.3) circle (0.1);
 \draw [rounded corners] (1,0.9) -- (1.3,1.2) -- (1.5,1.25);
 \draw [dashed, rounded corners] (2,0.8)-- (1.7,1.2) -- (1.5,1.25);
 \fill  (1.5,1.25) circle (0.1);
 \fill [teal] (1,0.8) circle (0.12);
 \fill [teal] (2,0.8) circle (0.1);
 \node at (1,0) [below] {$H \otimes V \otimes H^*$};
 \node  at (3.5,-2){ };
\end{tikzpicture}
   \captionof{figure}{cYBE for $H \otimes V \otimes H^*$ $\Longleftrightarrow$ \eqref{eqn:YD}}\label{pic:cYBE_YD}
\end{center}
\end{proof}

\subsection{Tensor product of Yetter-Drinfel$'$d modules}\label{sec:YD_tensor}

The aim of this section is to explain the definition \eqref{eqn:tensor_mod_alter}-\eqref{eqn:tensor_comod_alter} of tensor product of YD modules from the braided point of vue, using the braided interpretation of YD structure presented in Proposition \ref{thm:YDsyst_br_precision}.

Start with an easy general observation.

\begin{lemma}\label{thm:glueing_components}
Let $(\oV,\osigma)$ be a rank $4$ braided system in a monoidal category $\CC.$ A braiding can then be defined for the rank $3$ system $(V_1, V_2 \otimes V_3, V_4)$ in the following way:
\begin{itemize}
\item keep $\sigma_{V_1,V_1},$ $\sigma_{V_1,V_4}$ and $\sigma_{V_4,V_4}$ from the previous system;
\item put $\sigma_{V_2 \otimes V_3,V_2 \otimes V_3} = \Id_{(V_2 \otimes V_3) \otimes (V_2 \otimes V_3)}$;
\item define
\begin{align*}
\sigma_{V_1,V_2 \otimes V_3} &= (\Id_{V_2} \otimes \sigma_{1,3}) \otimes (\sigma_{1,2} \otimes \Id_{V_3}),\\
\sigma_{V_2 \otimes V_3,V_4} &= (\sigma_{2,4} \otimes \Id_{V_3}) \otimes (\Id_{V_2} \otimes \sigma_{3,4}).
\end{align*}
\end{itemize}
\end{lemma}

\begin{proof}
The instances of the cYBE not involving $V_2 \otimes V_3$ hold true since they were true in the original braided system. Those involving $V_2 \otimes V_3$ at least twice are trivially true. The remaining instances (those involving $V_2 \otimes V_3$ exactly once) are verified by applying the instances of the cYBE coming from the original braided system several times.
\end{proof}

Certainly, this \emph{gluing procedure} can be applied in a similar way to any two or more consecutive components in a braided system of any rank.

\medskip
Applying the gluing procedure to the braided system from Example \ref{ex:multi_YD_as_YD_system} ($r=2$) and slightly rewriting the braiding obtained (using the coassociativity of $H$ and of $H^*$), one gets

\begin{lemma}
Given YD modules $M_1$ and $M_2$ over a finite-dimensional $\k$-linear bialgebra $H,$ one has the following braiding on $(H,M_1 \otimes M_2, H^*)$:
\begin{align*}
\sigma_{H,H} &:= \sigma_{Ass}^r(H) = \mu_H \otimes \nu_H,\\
\sigma_{H^*,H^*} &:= \sigma_{Ass}(H^*) = (\varepsilon_H)^* \otimes (\Delta_H)^*,\\
\sigma_{H,H^*} &:= \crot^{YD}_{H,H^*},\\ 
\sigma_{M_1 \otimes M_2,M_1 \otimes M_2} &:= \Id_{(M_1 \otimes M_2) \otimes (M_1 \otimes M_2)},\\
\sigma_{H,M_1 \otimes M_2} &:= c_{H,M_1 \otimes M_2} \circ (\Id_H \otimes \mathring{\lambda}_{M_1 \otimes M_2}) \circ (\Delta_H \otimes \Id_{M_1 \otimes M_2}),\\ 
\sigma_{M_1 \otimes M_2,H^*} &:= c_{M_1 \otimes M_2,H^*} \circ (\Id_{M_1 \otimes M_2} \otimes \lambda_{H^*}) \circ (\mathring{\delta}_{M_1 \otimes M_2} \otimes \Id_{H^*}),
\end{align*}
where $\lambda_{H^*}$ denotes the action \eqref{eqn:H_acts_on_H*} of $H$ on $H^*,$ and $\mathring{\lambda}_{M_1 \otimes M_2}$ and $\mathring{\delta}_{M_1 \otimes M_2}$ are defined by \eqref{eqn:tensor_mod_alter} and \eqref{eqn:tensor_comod_alter}.
\end{lemma}

Now plug $\mathring{\lambda}_{M_1 \otimes M_2}$ and $\mathring{\delta}_{M_1 \otimes M_2}$ (and arbitrary $\mu$ and $\nu$) into  Proposition \ref{thm:YDsyst_br_precision}. The preceding lemma ensures the cYBE on $H \otimes (M_1 \otimes M_2) \otimes H^*,$ $H \otimes H \otimes (M_1 \otimes M_2)$ and $(M_1 \otimes M_2)\otimes H^* \otimes H^*$ (since the corresponding components of $\osigma$ from the proposition and from the lemma coincide). The equivalences from the proposition then allow to conclude:

\begin{corollary} 
Given YD modules $M_1$ and $M_2$ over a finite-dimensional $\k$-linear bialgebra $H,$ the morphisms $\mathring{\lambda}_{M_1 \otimes M_2}$ and $\mathring{\delta}_{M_1 \otimes M_2}$ defined by \eqref{eqn:tensor_mod_alter} and \eqref{eqn:tensor_comod_alter} endow $M_1 \otimes M_2$ with a YD module structure.
\end{corollary} 

One thus obtains a more conceptual way of ``guessing'' the correct definition of tensor product for YD modules.

\subsection{Braided homology: a short reminder}\label{sec:hom}

In this section we recall the homology theory for braided systems developed in \cite{Lebed2} (see also \cite{Lebed1} for the rank $1,$ i.e. braided object, case). In the next section we apply this general theory to the braided system constructed out of a YD module in Theorem \ref{thm:YDsystem_for_YDmod}.

We first explain what we mean by a \emph{homology theory} for a braided system $(\oV,\osigma)$ in an \underline{additive} monoidal category $\CC$:

\begin{definition}\label{def:DiffCat}
\begin{itemize}
\item  \emph{A degree $-1$ differential} for a collection $(X_n)_{n \ge 0}$ of objects in $\CC$ is a family of morphisms $(d_n: X_n \rightarrow X_{n-1} )_{n>0},$ satisfying 
\begin{align*}
d_{n-1}\circ d_n &= 0 & \forall \: n>1.
\end{align*}

\item \emph{A bidegree $-1$ bidifferential} for  a collection $(X_n)_{n \ge 0}$ of objects in $\CC$  consists of two families of morphisms $(d_n,d'_n: X_n \rightarrow X_{n-1} )_{n>0},$  satisfying
\begin{align*}
d_{n-1} \circ d_n &= d'_{n-1} \circ d'_n = d'_{n-1} \circ d_n +d_{n-1} \circ d'_n = 0 & \forall \: n>1.
\end{align*}

\item An \emph{ordered tensor product} for $(\oV,\osigma)\in \BrSyst_r(\CC)$ is a tensor product of the form
$$V_1^{\otimes m_1}\otimes V_2^{\otimes m_2} \otimes\cdots \otimes V_r^{\otimes m_r}, \qquad m_i \ge 0.$$
\item The \emph{degree} of such a tensor product is the sum $\sum_{i=1}^r m_i.$ 
\item The direct sum of all ordered tensor products of degree $n$ is denoted by $T(\oV)_n^\rightarrow.$

\item  {A (bi)degree $-1$ (bi)differential} for $(\oV,\osigma)$ is a (bi)degree $-1$ (bi)differential for $(T(\oV)_n^\rightarrow)_{n \ge 0}.$ 
\end{itemize}
\end{definition}

One also needs the ``braided'' notion of character, restricting in particular examples to the familiar notions of character for several algebraic structures such as UAAs and Lie algebras (cf. \cite{Lebed1}).

\begin{definition}\label{def:BrChar}
A \emph{braided character} for $(\oV,\osigma)\in \BrSyst_r(\CC)$ is a rank $r$ braided system morphism from $(\oV,\osigma)$ to $(\II, \ldots, \II; \sigma_{i,j} = \Id_{\II} \forall \: i<j).$
\end{definition}

In other words, it is a collection of morphisms $(\zeta_i:V_i \rightarrow \II)_{1 \le i \le r}$ satisfying the compatibility condition
\begin{align}
(\zeta_j \otimes \zeta_i) \circ \sigma_{i,j} &= \zeta_i \otimes \zeta_j & \forall i \le j.\label{eqn:br_char}
\end{align}

We now exhibit a bidegree $-1$ bidifferential for an arbitrary braided system endowed with braided characters; see \cite{Lebed2} for motivations, proofs, a multi-quantum shuffle interpretation and properties of this construction.

\begin{theorem}\label{thm:BraidedSimplHomCat}
Take a braided system $(\oV,\osigma)$ in an additive monoidal category $\CC,$ equipped with two braided characters $\overline{\zeta}$ and $\overline{\xi}.$ The families of morphisms 
\begin{align*}
({^{\zeta}}\! d)_n &:= \sum_{i=1}^n  (\zeta_*)^1 \circ (-\sigma_{*,*})^1 \circ (-\sigma_{*,*})^2 \circ \cdots \circ (-\sigma_{*,*})^{i-1},\\
(d^{\xi})_n &:= (-1)^{n-1} \sum_{i=1}^n  (\xi_*)^n \circ (-\sigma_{*,*})^{n-1} \circ \cdots \circ (-\sigma_{*,*})^{i+1} \circ (-\sigma_{*,*})^i.
\end{align*}
from $T(\oV)_{n}^\rightarrow$ to $T(\oV)_{n-1}^\rightarrow$ define a bidegree $-1$ tensor bidifferential. Here the stars $*$ mean that each time one should choose the component of $\osigma, \overline{\zeta}$ or $\overline{\xi}$ corresponding to the $V_k$'s on which it acts. Further, notation \eqref{eqn:phi_i} for superscripts is used.
\end{theorem}

Pictorially, $({^\zeta}\! d)_n$ for example is a signed sum (due to the use of the negative braiding $-\osigma$) of the terms presented on Fig. \ref{pic:BrDiffCoeffs}. The \emph{sign} can be interpreted via the intersection number of the diagram.

\begin{center}
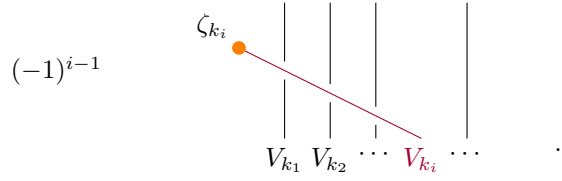

\begin{tikzpicture}[scale=0.6]
 \node at (-4,1.5) {$(-1)^{i-1}$};
 \draw (1,0) -- (1,1.3);
 \draw (1,1.7) -- (1,3);
 \draw (2,0) -- (2,0.8);
 \draw (2,1.2) -- (2,3);
 \draw (3,0) -- (3,0.3);
 \draw (3,0.7) -- (3,3);
 \draw [colori] (4,0) -- (0,2);
 \draw (5,0) -- (5,3);
 \node at (0,2) [above left]{$\zeta_{k_i}$};
 \fill [orange] (0,2) circle (0.15);
 \node at (1,0) [below] {$V_{k_1}$};
 \node at (2,0) [below] {$V_{k_2}$};
 \node at (3,0) [below] {$\cdots$};
 \node at (4,0) [colori,below] {$V_{k_i}$};
 \node at (5,0) [below] {$\cdots$};
 \node at (7,0) [below] {.};
\end{tikzpicture}
   \captionof{figure}{Multi-braided left differential}\label{pic:BrDiffCoeffs}
\end{center}

\begin{corollary}
Any $\ZZ$-linear combination of the families $({^{\zeta}}\! d)_n$ and $(d^{\xi})_n$ from the theorem is a degree $-1$ differential.
\end{corollary}

\begin{definition}
The \emph{(bi)differentials} from the above theorem and corollary are called \emph{multi-braided}.
\end{definition}

\begin{remark}
\begin{itemize}
\item The constructions from the theorem are \emph{functorial}.
\item Applying the categorical duality to this theorem, one gets a \emph{\underline{co}homology theory} for $(\oV,\osigma).$ 
\item The braided bidifferentials can be shown to come from a structure of \emph{simplicial} (or, more precisely, \emph{cubical}) type.
\end{itemize}
\end{remark}

\subsection{Braided homology of Yetter-Drinfel$'$d modules}\label{sec:YD_hom}

Let us now apply Theorem \ref{thm:BraidedSimplHomCat} to the braided system from Theorem \ref{thm:YDsystem_for_YDmod}. As for braided characters, we use the following ones:

\begin{lemma}
Morphisms $(\varepsilon_H: H \rightarrow \k, \; 0: M  \rightarrow \k, \; 0: H^* \rightarrow \k)$ and $(0: H \rightarrow \k, \;  0: M  \rightarrow \k, \; \varepsilon_{H^*} = \nu_H^*: H^* \rightarrow \k)$ are braided charcters for the braided system from Theorem \ref{thm:YDsystem_for_YDmod}. 
\end{lemma}

These braided characters are abusively denoted by $\varepsilon_H$ and $\varepsilon_{H^*}$ in what follows.

\begin{proof}
We prove the statement for $\varepsilon_H$ only, the one for $\varepsilon_{H^*}$ being similar.

Both sides of \eqref{eqn:br_char} are identically zero for the morphisms from the lemma, except for the case $i=j=1,$ corresponding to $H \otimes H.$ In this latter case, \eqref{eqn:br_char} becomes
$$ (\varepsilon_H \otimes \varepsilon_H) \circ (\mu_H \otimes \nu_H) = \varepsilon_H \otimes \varepsilon_H,$$
which follows from the fact that $\varepsilon_H$ is an algebra morphism (cf. the definition of bialgebra).
\end{proof}

In what follows, the letters $h_i$ always stay for elements of $H,$ $l_j$  -- for elements of $H^*;$ $a \in M,$ $b \in N^*;$ the pairing $\left\langle,\right\rangle$ is the evaluation; the multiplications $\mu_H$ and $\Delta_H^*$ on $H$ and $H^*$ respectively, as well as $H$- and $H^*$-actions, are denoted by $\cdot$ for simplicity. We also use higher-order Sweedler's notations of type
\begin{align*}
(\Delta_H \otimes \Id_H) \circ \Delta_H (h) &= h_{(1)} \otimes h_{(2)} \otimes h_{(3)}, & \forall \:  h \in H.
\end{align*}
Further, we omit the tensor product sign when it does not lead to confusion, writing for instance $h_1\ldots h_n \in H^{\otimes n}.$

Using these notations, one can write down explicite braided differentials for a YD module:

\begin{proposition}\label{thm:YDmod2}
Take a Yetter-Drinfel$'$d module $(M,\lambda,\delta)$ over a finite-dimensional $\k$-linear bialgebra $H.$ There is a bidegree $-1$ bidifferential on $T(H)\otimes M \otimes T(H^*)$ given by
\begin{align*}
{^{\varepsilon_{H^*}}}\! d&(h_1\ldots h_n\otimes a\otimes l_1\ldots l_m) =\\
&(-1)^{n+1}\left\langle l_{1(1)},a_{(1)}\right\rangle \left\langle l_{1(2)},h_{n(2)}\right\rangle \left\langle  l_{1(3)},h_{n-1(2)}\right\rangle \ldots \left\langle l_{1(n+1)},h_{1(2)}\right\rangle   h_{1(1)}\ldots h_{n(1)} \otimes a_{(0)}\otimes l_2\ldots l_m \\
&+\sum_{i=1}^{m-1} (-1)^{n+i+1}h_1\ldots h_n\otimes a\otimes l_1\ldots l_{i-1}(l_i\cdot l_{i+1})l_{i+2}\ldots l_m,\\
d{^{\varepsilon_{H}}}&(h_1\ldots h_n\otimes a\otimes l_1\ldots l_m) =\\
&(-1)^{n-1}\left\langle l_{1(1)},h_{n(m)}\right\rangle \left\langle l_{2(1)},h_{n(m-1)}\right\rangle \ldots \left\langle l_{m(1)},h_{n(1)}\right\rangle
 h_1\ldots h_{n-1} \otimes (h_{n(m+1)}\cdot a)\otimes l_{1(2)}\ldots l_{m(2)}\\
&+\sum_{i=1}^{n-1} (-1)^{i-1}h_1\ldots h_{i-1}(h_i \cdot h_{i+1})h_{i+2}\ldots h_n\otimes a\otimes l_1\ldots l_m.
\end{align*}
\end{proposition}

\begin{proof}
These are just the constructions from Theorem \ref{thm:BraidedSimplHomCat} applied to the braided system from Theorem \ref{thm:YDsystem_for_YDmod} and braided characters $\varepsilon_{H^*}$ and $\varepsilon_H.$ Note that we restricted the differentials from $\oplus_{n=0}^\infty T(\oV)_{n}$ to $T(H)\otimes M \otimes T(H^*),$ which is possible since both $\varepsilon_H$ and $\varepsilon_{H^*}$ are zero on $M.$
\end{proof}

\begin{remark}
In fact the differentials ${^{\varepsilon_{H^*}}}\! d$ and $d{^{\varepsilon_{H}}}$ define a \emph{double complex} structure on $T(H) \otimes M \otimes T(H^*),$ graduated by putting
$$\overline{\deg} \: (H^{\otimes n} \otimes M \otimes (H^*)^{\otimes m}) = (n,m).$$
\end{remark}

\medskip
In order to get rid of the classical contracting homotopies of type
\begin{align*}
\overline{h} \otimes a\otimes \overline{l} &\longmapsto 1_H \overline{h} \otimes a\otimes \overline{l},\\
\overline{h} \otimes a\otimes \overline{l} &\longmapsto \overline{h} \otimes a\otimes \overline{l}1_{H^*},
\end{align*}
we now try to ``cycle'' this bidifferential, in the spirit of Hochschild homology for algebras or Gerstenhaber-Schack homology for bialgebras.

Concretely, take a {YD module} $(M,\lambda_M,\delta_M)$ and a {finite-dimensional YD module} $(N,\lambda_N,\delta_N)$ over $H.$ Our aim is to endow the graded vector space 
$$T(H) \otimes M \otimes T(H^*) \otimes N^*$$ 
with a bidifferential extending that from Proposition \ref{thm:YDmod2}. 

First, note that the ``rainbow'' duality between $H$ and $H^*$ (cf. \eqref{eqn:Delta_dual} or Fig. \ref{pic:DualViaRainbow}) graphically corresponds to an angle $\pi$ rotation. Since the notions of bialgebra and YD module are centrally symmetric (cf. Remark \ref{rmk:YD_br_alter}), one gets the following useful property:

\begin{lemma}\label{thm:YD_dual}
Take a {finite-dimensional YD module} $(N,\lambda_N,\delta_N)$ over a finite dimensional bialgebra $H.$ Then $(N^*,\delta_N^*,\lambda_N^*)$ is a YD module over $H^*.$
\end{lemma}

Further, inspired by the formulas from Proposition \ref{thm:YDmod2}, consider the following morphisms from $ H^{\otimes n}\otimes M \otimes (H^*)^{\otimes m} \otimes N^*$ to $ H^{\otimes n-1}\otimes M \otimes (H^*)^{\otimes m} \otimes N^*$ or $ H^{\otimes n}\otimes M \otimes (H^*)^{\otimes m-1} \otimes N^*$:
\begin{align*}
^{H^*}\!\pi &(h_1\ldots h_n\otimes a\otimes l_1\ldots l_m \otimes b) =\\
&\left\langle l_{1(1)},a_{(1)}\right\rangle \left\langle l_{1(2)},h_{n(2)}\right\rangle \left\langle  l_{1(3)},h_{n-1(2)}\right\rangle \ldots \left\langle l_{1(n+1)},h_{1(2)}\right\rangle h_{1(1)}\ldots h_{n(1)} \otimes a_{(0)}\otimes l_2\ldots l_m\otimes  b =\\
&\left\langle l_{1}, h_{1(2)}  \cdots h_{n-1 (2)} \cdot h_{n(2)} \cdot a_{(1)}\right\rangle 
h_{1(1)} \ldots h_{n(1)} \otimes a_{(0)} \otimes l_2\ldots l_m\otimes  b, \\
\pi^{H^*} &(h_1\ldots h_n\otimes a\otimes l_1\ldots l_m \otimes b) =\\
&\left\langle l_{m(1)}, h_{1(1)} \cdot h_{2(1)} \cdot \cdots \cdot h_{n(1)} \right\rangle 
h_{1(2)} \ldots h_{n(2)} \otimes a \otimes l_1\ldots l_{m-1} \otimes l_{m(2)} \cdot b, \\
 \pi^{H}&(h_1\ldots h_n\otimes a\otimes l_1\ldots l_m \otimes b) =\\
&\left\langle l_{1(1)},h_{n(m)}\right\rangle \left\langle l_{2(1)},h_{n(m-1)}\right\rangle \ldots \left\langle l_{m(1)},h_{n(1)}\right\rangle h_1\ldots h_{n-1} \otimes (h_{n(m+1)}\cdot a)\otimes l_{1(2)}\ldots l_{m(2)} \otimes b=\\
&\left\langle l_{1(1)} \cdot l_{2(1)} \cdots l_{m(1)},h_{n(1)} \right\rangle h_1\ldots h_{n-1} \otimes (h_{n(2)}\cdot a)\otimes l_{1(2)}\ldots l_{m(2)} \otimes b,\\
 ^{H}\!\pi &(h_1\ldots h_n\otimes a\otimes l_1\ldots l_m \otimes b) =\\
&\left\langle l_{1(2)} \cdot l_{2(2)} \cdots l_{m(2)} \cdot b_{(1)},h_{1}\right\rangle h_2\ldots h_{n} \otimes a\otimes l_{1(1)}\ldots l_{m(1)} \otimes b_{(0)}.\\
\end{align*}
These applications are presented on Fig. \ref{pic:YD_hom_components}. We use notation
\begin{align*}
\Delta^t &:= \Delta_H^t := (\Delta_H \otimes \Id_{H^{\otimes t-1}}) \circ \cdots \circ (\Delta_H \otimes \Id_{H}) \circ \Delta_H:H \rightarrow H^{\otimes \: t+1} & \forall \: t \in \NN, 
\end{align*}
and similarly for $(\mu^*)^t.$

\begin{center}
\begin{tikzpicture}[scale=0.7]
 \draw [dashed] (0,0) --  (0,4);
 \draw [dashed] (1,0) -- (1,4);
 \draw [dashed] (2,0) -- (2,4);
 \draw  (-3,0) -- (-3,4);
 \draw  (-4,0) -- (-4,4);
 \draw [thick, colori] (3,0) -- (3,4);
 \draw [thick, colorM] (-1,0) -- (-1,4);
 \draw (-0.5,1.5) -- (-2,0.5) -- (-2,0);
 \draw (-0.5,2.5) -- (-2,0.5) -- (-0.5,3.5);
 \draw (-2,0.5) -- (-1,0.8);
 \draw [dashed] (0,0.5) -- (-0.5,1.5);
 \draw [dashed] (1,0.5) -- (-0.5,2.5);
 \draw [dashed] (2,0.5) -- (-0.5,3.5);
 \node at (-5,2)  {$\pi^{H} =$};
 \node at (0,0.5) [right] {$\mu^*$};
 \node at (1,0.5) [right] {$\mu^*$};
 \node at (2,0.5) [right] {$\mu^*$};
 \node at (-1,0.8) [right] {$\lambda_M$};
 \node at (-2.3,0.5) [above] {$\Delta^{m}$};
 \node at (-0.5,1.5) [above] {$ev$};
 \node at (-0.5,2.5) [above] {$ev$};
 \node at (-0.5,3.5) [above] {$ev$};
 \fill [teal] (0,0.5) circle (0.1);
 \fill [teal] (1,0.5) circle (0.1);
 \fill [teal] (2,0.5) circle (0.1);
 \fill [teal] (-2,0.5) circle (0.1);
 \fill [teal] (-1,0.8) circle (0.15);
 \fill (-0.5,1.5) circle (0.1);
 \fill (-0.5,2.5) circle (0.1);
 \fill (-0.5,3.5) circle (0.1);
 \node at (3.5,0) {,};
 \node at (-3,0) [below] {$H^{\otimes n}$};
 \node at (1,0) [below] {$(H^*)^{\otimes m}$};
 \node at (-1,0) [below] {$M$};
 \node at (3,0) [below] {$N^*$};
\end{tikzpicture}
\begin{tikzpicture}[scale=0.7]
 \node at (-1,2)  {$\qquad ^{H^*}\!\pi =$};
 \draw (1,0) -- (1,4);
 \draw (2,0) -- (2,4);
 \draw (3,0) -- (3,4);
 \draw [thick, colorM] (4,0) -- (4,4);
 \draw [dashed] (6,0) -- (6,4);
 \draw [dashed] (7,0) -- (7,4);
 \draw [thick, colori] (8,0) -- (8,4);
 \draw[dashed] (4.5,1.5) -- (5,0.5) -- (5,0);
 \draw[dashed] (4.5,2.5) -- (5,0.5) -- (4.5,3.5);
 \draw[dashed] (4.5,0.8) -- (5,0.5);
 \draw (1,0.5) -- (4.5,3.5);
 \draw (2,0.5) -- (4.5,2.5);
 \draw (3,0.5) -- (4.5,1.5);
 \draw (4,0.5) -- (4.5,0.8);
 \node at (1,0.5) [left] {$\Delta$};
 \node at (2,0.5) [left] {$\Delta$};
 \node at (3,0.5) [left] {$\Delta$};
 \node at (4,0.5) [left] {$\delta_M$};
 \node at (4.8,0.5) [above right] {$(\mu^*)^{n}$};
 \fill [teal] (1,0.5) circle (0.1);
 \fill [teal] (2,0.5) circle (0.1);
 \fill [teal] (3,0.5) circle (0.1);
 \fill [teal] (4,0.5) circle (0.15);
 \fill [teal] (5,0.5) circle (0.1);
 \fill (4.5,0.8) circle (0.1);
 \fill (4.5,1.5) circle (0.1);
 \fill (4.5,2.5) circle (0.1);
 \fill (4.5,3.5) circle (0.1);
 \node at (8.5,0) {,};
 \node at (8,0) [below] {$N^*$};
 \node at (6,0) [below] {$(H^*)^{\otimes m}$};
 \node at (4,0) [below] {$M$};
 \node at (2,0) [below] {$H^{\otimes n}$};
\end{tikzpicture}

\begin{tikzpicture}[scale=0.7]
 \node at (-2,2)  {$^{H}\!\pi =$};
 \draw [dashed] (1,-0.5) -- (1,4);
 \draw [dashed] (2,-0.5) -- (2,4);
 \draw [dashed] (3,-0.5) -- (3,4);
 \draw [thick, colori] (4,-0.5) -- (4,4);
 \draw (-1,-0.5) -- (-1,4);
 \draw (-0.5,-0.5) -- (-0.5,4);
 \draw [thick, colorM] (0,-0.5) -- (0,4);
 \draw [rounded corners] (5,0.5) -- (4.5,0.3) -- (-1,-0.3) -- (-1.5,-0.5);
 \draw (4.5,1.5) -- (5,0.5);
 \draw (4.5,2.5) -- (5,0.5) -- (4.5,3.5);
 \draw (4.5,0.8) -- (5,0.5);
 \draw [dashed] (1,0.5) -- (4.5,3.5);
 \draw [dashed] (2,0.5) -- (4.5,2.5);
 \draw [dashed] (3,0.5) -- (4.5,1.5);
 \draw [dashed] (4,0.5) -- (4.5,0.8);
 \node at (1,0.5) [left] {$\mu^*$};
 \node at (2,0.5) [left] {$\mu^*$};
 \node at (3,0.5) [left] {$\mu^*$};
 \node at (4.2,0.7) [left] {$\delta_{N^*}$};
 \node at (4.8,0.5) [above right] {$\Delta^{m}$};
 \fill [teal] (1,0.5) circle (0.1);
 \fill [teal] (2,0.5) circle (0.1);
 \fill [teal] (3,0.5) circle (0.1);
 \fill [teal] (4,0.5) circle (0.15);
 \fill [teal] (5,0.5) circle (0.1);
 \fill (4.5,0.8) circle (0.1);
 \fill (4.5,1.5) circle (0.1);
 \fill (4.5,2.5) circle (0.1);
 \fill (4.5,3.5) circle (0.1);
 \node at (5,-0.5) {,};
 \node at (4,-0.5) [below] {$N^*$};
 \node at (2,-0.5) [below] {$(H^*)^{\otimes m}$};
 \node at (0,-0.5) [below] {$M$};
 \node at (-1,-0.5) [below] {$H^{\otimes n}$};
\end{tikzpicture} 
\begin{tikzpicture}[scale=0.7]
 \draw (0,-0.5) --  (0,4);
 \draw (1,-0.5) -- (1,4);
 \draw (2,-0.5) -- (2,4);
 \draw [dashed] (5,-0.5) -- (5,4);
 \draw [dashed] (4,-0.5) -- (4,4);
 \draw [thick, colorM] (3,-0.5) -- (3,4);
 \draw [thick, colori] (7,-0.5) -- (7,4);
 \draw [rounded corners, dashed] (-1,1) -- (-0.8,0.4) -- (6,0) -- (6,-0.5);
 \draw [dashed] (7,1) -- (6,0);
 \draw [dashed] (-0.5,1.5) -- (-1,1);
 \draw [dashed] (-0.5,2.5) -- (-1,1) -- (-0.5,3.5);
 \draw (0,0.5) -- (-0.5,1.5);
 \draw (1,0.5) -- (-0.5,2.5);
 \draw (2,0.5) -- (-0.5,3.5);
 \node at (-3,2)  {$\qquad\pi^{H^*} =$};
 \node at (0,0.5) [right] {$\Delta$};
 \node at (1,0.5) [right] {$\Delta$};
 \node at (2,0.5) [right] {$\Delta$};
 \node at (7,1) [above left] {$\lambda_{N^*}$};
 \node at (-1.5,1) [below] {$(\mu^*)^{n-1}$};
 \node at (6,0) [above] {$\mu^*$};
 \fill [teal] (0,0.5) circle (0.1);
 \fill [teal] (1,0.5) circle (0.1);
 \fill [teal] (2,0.5) circle (0.1);
 \fill [teal] (-1,1) circle (0.1);
 \fill [teal] (6,0) circle (0.1);
 \fill [teal] (7,1) circle (0.15);
 \fill (-0.5,1.5) circle (0.1);
 \fill (-0.5,2.5) circle (0.1);
 \fill (-0.5,3.5) circle (0.1);
 \node at (7.5,-0.5) {.};
 \node at (1,-0.5) [below] {$H^{\otimes n}$};
 \node at (5,-0.5) [below] {$(H^*)^{\otimes m}$};
 \node at (3,-0.5) [below] {$M$};
 \node at (7,-0.5) [below] {$N^*$};
\end{tikzpicture}
  
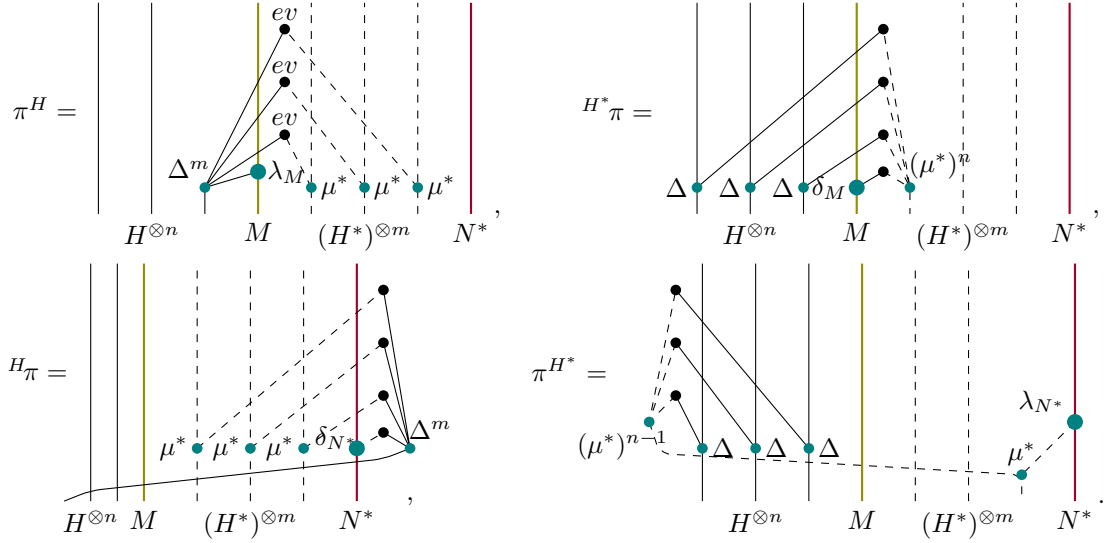
\captionof{figure}{Components of YD module homology}\label{pic:YD_hom_components}
\end{center}

These morphisms can be interpreted in terms of ``braided'' \emph{adjoint actions}, as it was done for the bialgebra case in \cite{Lebed2}.

At last, recall the classical \emph{bar and cobar differentials}: 
\begin{align*}
d_{bar}&(h_1\ldots h_n \otimes a \otimes l_1\ldots l_m \otimes b) = \sum_{i=1}^{n-1} (-1)^{i}h_1\ldots h_{i-1}(h_i \cdot h_{i+1})h_{i+2}\ldots h_n \otimes a \otimes l_1\ldots l_m \otimes b,\\
d_{cob}&(h_1\ldots h_n \otimes a \otimes l_1\ldots l_m \otimes b) = \sum_{i=1}^{m-1} (-1)^{i}h_1\ldots h_n \otimes a \otimes l_1\ldots l_{i-1}(l_i\cdot l_{i+1})l_{i+2}\ldots l_m \otimes b.
\end{align*}

Everything is now ready for presenting an enhanced version of Proposition \ref{thm:YDmod2}:

\begin{theorem}\label{thm:YDHomGen}
Given a {YD module} $M$ and a {finite-dimensional YD module} $N$ over a {finite-dimensional} {bialgebra} $H$ in $\Vect,$ one has four bidegree $-1$ bidifferentials on $T(H) \otimes M \otimes T(H^*) \otimes N^*,$ presented in the lines of Table \ref{tab:YDModBidiff}.
\begin{center}
\begin{tabular}{|c||c|c|}
\hline 
1.&$ d_{bar}$&$(-1)^{n}d_{cob}$ \\
\hline 
2. &$d_{bar}+(-1)^{n}\pi^{H}$ &$(-1)^{n}d_{cob}+(-1)^{n}({^{H^*}}\!\pi)$ \\
\hline 
3.& $d_{bar}+{^{H}}\!\pi$& $(-1)^{n}d_{cob}+(-1)^{n+m}{\pi^{H^*}}$\\
\hline 
4.&$d_{bar}+(-1)^{n}\pi^{H}+{^{H}}\!\pi$ & $(-1)^{n}d_{cob}+(-1)^{n}({^{H^*}}\!\pi)+(-1)^{n+m}{\pi^{H^*}}$\\
\hline 
\end{tabular}
   \captionof{table}{Bidifferential structures on $T(H) \otimes M \otimes T(H^*) \otimes N^*$}\label{tab:YDModBidiff}
\end{center}
The signs $(-1)^{n}$ etc. here are those one chooses on the component $ H^{\otimes n} \otimes M \otimes (H^*)^{\otimes m} \otimes N^*$ of $T(H) \otimes M \otimes T(H^*) \otimes N^*.$
\end{theorem}

Substituting the graded vector space $T(H) \otimes M \otimes T(H^*) \otimes N^*$ we work in with its alternative version  $\Hom_\k(N \otimes T(H), T(H) \otimes M)$ (as it was done for example in \cite{MW}), we obtain (the dual of a mirror version of) the \emph{deformation cohomology for YD modules}, defined by F. Panaite and D. {\c{S}}tefan in \cite{Panaite}. We have thus developed a conceptual framework for this cohomology theory, replacing case by case verifications (for instance, when proving that one has indeed a bidifferential) with a structure study, facilitated by graphical tools.

\subsection{Proof of Theorem \ref{thm:YDHomGen}}

Morphisms $d_{bar}$ and $d_{cob}$ are well known to be differentials; this can be easily verified by direct calculations, or using the rank $1$ braided systems $(H, \sigma_{Ass}(H))$ and $(H^*, \sigma_{Ass}(H^*))$, cf. \cite{Lebed1}. Further, they affect different components of $T(H) \otimes M \otimes T(H^*) \otimes N^*$ (namely, $T(H)$ and $T(H^*)$) and thus commute. The sign $(-1)^{n}$ guarantees the anti-commutation. This proves that the first line of the table contains a bidegree $-1$ bidifferential.

The assertion for the second line follows from Proposition \ref{thm:YDmod2}, since
\begin{align*}
{^{\varepsilon_{H^*}}}\! d \otimes \Id_{N^*} &= (-1)^{n+1} d_{cob} + (-1)^{n+1} ({^{H^*}}\!\pi),\\
d{^{\varepsilon_{H}}}  \otimes \Id_{N^*} &=  - d_{bar} + (-1)^{n-1}\pi^{H}.
\end{align*}

To prove the statement for the third line, note that the dual (in the sense of Lemma \ref{thm:YD_dual}) version of Proposition \ref{thm:YDmod2} gives a bidifferential $({^{\varepsilon_{H}}}\!d,d{^{\varepsilon_{H^*}}})$ on $T(H^*) \otimes N^* \otimes T(H),$ and hence (by tensoring with $\Id_M$) on  $T(H^*) \otimes N^* \otimes T(H) \otimes M.$ Identifying the latter space with $T(H) \otimes M \otimes T(H^*) \otimes N^*$ via the flip \eqref{eqn:flip} and keeping notation $({^{\varepsilon_{H}}}\!d,d{^{\varepsilon_{H^*}}})$ for the bidifferential induced on this new space, one calculates
\begin{align*}
{^{\varepsilon_{H}}}\!d &= (-1)^{m+1} d_{bar} + (-1)^{m+1} ({}^H\!\pi),\\
d{^{\varepsilon_{H^*}}}  &= - d_{cob} + (-1)^{m-1} \pi^{H^*}.
\end{align*}
Multiplying ${^{\varepsilon_{H}}}\! d$ by $(-1)^{m+1}$ and $d{^{\varepsilon_{H^*}}}$ by $(-1)^{n+1},$ and checking that this does not break the anti-commutation, one recovers the third line of the table.

We start the proof for the fourth line with a general assertion:

\begin{lemma}\label{thm:AlgTrick}
Take an abelian group $(S,+,0,x\mapsto -x)$ endowed with an operation $\cdot$ distributive with respect to $+.$ Then, for any $a,b,c,d,e,f \in S$ such that
$$(a+b)\cdot (d+e) = (a+c)\cdot (d+f) = a \cdot d = b \cdot f + c \cdot e = 0,$$
 one has 
$$(a+b+c)\cdot (d+e+f) = 0.$$
\end{lemma}
\begin{proof}
$$(a+b+c)\cdot (d+e+f) = (a+b)\cdot (d+e) + (a+c)\cdot (d+f) - a \cdot d +( b \cdot f + c \cdot e).$$ \qedhere
\end{proof}
Now take $S=\End_\k(T(H) \otimes M \otimes T(H^*) \otimes N^*)$ with the usual addition and the operation composition $\varphi \circ \psi$ (for proving that the two morphisms from the fourth line of our table are differentials), or the operation $\varphi \diamond \psi:= \varphi \circ \psi + \psi \circ \varphi$ (for proving that the two morphisms anti-commute). The information from the first three lines and Lemma \ref{thm:AdjActionsForYDModCommute} allow to apply Lemma \ref{thm:AlgTrick} to sextuples $(a,b,c,a,b,c)$ and $(d,e,f,d,e,f)$ (for the operation $\circ$), and $(a,b,c,d,e,f)$ (for the operation $\diamond$), where
\begin{align*}
a &:= d_{bar}, & d &:= (-1)^{n}d_{cob},\\
b &:= (-1)^{n}\pi^{H}, & e &:= (-1)^{n}({^{H^*}}\!\pi)\\
c &:= {^{H}}\!\pi, & f &:= (-1)^{n+m}{\pi^{H^*}}.
\end{align*}
One thus gets the fourth line of the table.

\begin{lemma}\label{thm:AdjActionsForYDModCommute}
The endomorphisms ${^{H^*}}\!\pi,\pi^{H^*},\pi^{H}$ and ${^{H}}\!\pi$ of $T(H)\otimes M \otimes T(H^*) \otimes N^*$ pairwise commute.
\end{lemma}

\begin{proof}
The commutation of $\pi^{H}$ and ${^{H}}\!\pi$ follows from the coassociativity of $\mu^*$ (this is best seen on Fig. \ref{pic:YD_hom_components}). The pair ${^{H^*}}\!\pi,\pi^{H^*}$ is treated similarly.

The commutation of $\pi^{H}$ and ${^{H^*}}\!\pi$ follows from their interpretation as parts of a precubical structure (cf. \cite{Lebed1}), or by a direct computation. The pair ${^{H}}\!\pi,\pi^{H^*}$ is treated similarly.

The case of the pair $\pi^{H},\pi^{H^*}$ demands more work.

Denote by $\widehat{\pi}^{H}$ a version of $\pi^{H}$ which ``forgets'' the rightmost component of $H^*$:
\begin{align*}
 \widehat{\pi}^{H}&(h_1\ldots h_n\otimes a\otimes l_1\ldots l_m \otimes b) =\\
&\left\langle l_{1(1)} \cdot l_{2(1)} \cdots l_{m-1(1)},h_{n(1)} \right\rangle h_1\ldots h_{n-1} \otimes (h_{n(2)}\cdot a)\otimes l_{1(2)}\ldots l_{m-1(2)} l_m \otimes b.
\end{align*}
Similarly, denote by $\widehat{\pi}^{H^*}$ a version of $\pi^{H^*}$ which ``forgets'' the rightmost component of $H$:
\begin{align*}
\widehat{\pi}^{H^*} &(h_1\ldots h_n\otimes a\otimes l_1\ldots l_m \otimes b) =\\
&\left\langle l_{m(1)}, h_{1(1)} \cdot h_{2(1)} \cdot \cdots \cdot h_{n-1(1)} \right\rangle 
h_{1(2)} \ldots h_{n-1(2)} h_n \otimes a \otimes l_1\ldots l_{m-1} \otimes l_{m(2)} \cdot b.
\end{align*} 
Further, put
\begin{align*}
\theta &(h_1\ldots h_n\otimes a\otimes l_1\ldots l_m \otimes b) = \left\langle l_{m(1)}, h_{n(1)} \right\rangle 
h_1\ldots h_{n(2)}\otimes a\otimes l_1\ldots l_{m(2)} \otimes b.
\end{align*} 
Observe that $\theta$ is precisely the difference between the $\pi$'s and their reduced versions $\widehat{\pi}$:
\begin{align*}
 {\pi}^{H} &=  \widehat{\pi}^{H} \circ \theta, & {\pi}^{H^*} &=  \widehat{\pi}^{H^*} \circ \theta.
\end{align*} 
Moreover, one has
\begin{align*}
 {\pi}^{H} \circ  \widehat{\pi}^{H^*}  &= {\pi}^{H^*} \circ  \widehat{\pi}^{H},
\end{align*} 
since ${\pi}^{H}$ and $\widehat{\pi}^{H^*}$ modify different components of $T(H)\otimes M \otimes T(H^*) \otimes N^*$ (the hat $\;\widehat{}\;$ (dis)appears when these morphisms switch because ${\pi}^{H}$  kills the rightmost copy of $H,$ and ${\pi}^{H^*}$  kills the rightmost copy of $H^*$). Therefore,
\begin{align*}
{\pi}^{H^*} \circ  {\pi}^{H} &= {\pi}^{H^*} \circ  \widehat{\pi}^{H} \circ \theta = 
  {\pi}^{H} \circ  \widehat{\pi}^{H^*} \circ \theta =  {\pi}^{H} \circ  {\pi}^{H^*},
\end{align*} 
hence the desired commutation.

The pair ${^{H}}\!\pi,{^{H^*}}\!\pi$ is treated similarly.
\end{proof}

\begin{remark}
Lemma \ref{thm:AdjActionsForYDModCommute} actually contains more than needed for the proof of the theorem. We prefer keeping its full form and checking the commutation of all the $\binom{4}{2}=6$ pairs of morphisms for completeness.
\end{remark}

\subsection*{Acknowledgements}

The author is grateful to Marc Rosso and J{\'o}zef Przytycki for stimulating discussions.

\bibliographystyle{alpha}
\bibliography{biblio}

\end{document}